\documentclass[10pt]{amsart}
\usepackage{amsthm}
\usepackage{amsmath}
\usepackage{enumitem}
\usepackage{mathtools}
\usepackage{moreenum}
\usepackage{amssymb,verbatim, bbold}
\usepackage{upgreek}
\usepackage{xspace,cmap}
\usepackage{csquotes}
\usepackage{stmaryrd} % to get the proper \bigtriangleup operator for diagonal intersection
\usepackage[colorlinks,citecolor=blue,urlcolor=blue,linkcolor=blue]{hyperref}

\usepackage{mathtools}

\usepackage{pst-node}
\usepackage{tikz-cd}

\renewcommand{\restriction}{\mathbin\upharpoonright}    % by default in amssymb it's mathrel
\newtheorem*{theorem*}{Theorem}
\newtheorem*{maintheorem*}{Main Theorem}
\newtheorem*{corollary*}{Corollary}
\newtheorem*{definition*}{Definition}

\newtheorem{theorem}{Theorem}[section]

\newtheorem{prop}[theorem]{Proposition}
\newtheorem{claim}{Claim}[theorem]

\newtheorem{lemma}[theorem]{Lemma}
\newtheorem{cor}[theorem]{Corollary}

\newtheorem{question}{Question}
\newtheorem{fact}[theorem]{Fact}

\theoremstyle{definition}

\newtheorem{definition}[theorem]{Definition}

\theoremstyle{remark}
\newtheorem{remark}[theorem]{Remark}

\newcommand\HOD{\textnormal{HOD}}
\newcommand\OD{\textnormal{OD}}
\newcommand\ap{\textnormal{AP}}

\newcommand\cat[1]{{}^\curvearrowright #1}
\newcommand{\UEP}{\vec{\mathcal{U}}\textsf{-EP}}
\newcommand{\UBP}{\vec{\mathcal{U}}\textsf{-BP}}

\newcommand*\axiomfont[1]{\textsf{\textup{#1}}}
\newcommand\zfc{\axiomfont{ZFC}}
\newcommand\zf{\axiomfont{ZF}}
\newcommand\ac{\axiomfont{AC}}
\newcommand\ad{\axiomfont{AD}}
\newcommand\dc{\axiomfont{DC}}

\newcommand\gch{\axiomfont{GCH}}
\newcommand\sch{\axiomfont{SCH}}
\newcommand\psp{\axiomfont{PSP}}
\newcommand\bp{\axiomfont{BP}}

\newcommand\AP{\axiomfont{AP}}
\newcommand\ads{\axiomfont{ADS}}

\newcommand\npt{\axiomfont{NPT}}

\newcommand\ale[1]{\marginpar{Alejandro: #1}}
\newcommand\seba[1]{\marginpar{Seba: #1}}

\DeclareMathOperator{\supp}{supp}
\DeclareMathOperator{\crit}{crit}

\DeclareMathOperator{\ob}{OB}

    \def\sq{\sqsubseteq}
    
    \newcommand{\one}{\mathop{1\hskip-3pt {\rm l}}}

\makeatletter
\newcommand{\tpitchfork}{%
  \vbox{
    \baselineskip\z@skip
    \lineskip-.52ex
    \lineskiplimit\maxdimen
    \m@th
    \ialign{##\crcr\hidewidth\smash{$-$}\hidewidth\crcr$\pitchfork$\crcr}
  }%
}
\makeatother

\def\s{\subseteq}
\def\forces{\Vdash}
\DeclareMathOperator{\otp}{otp}

\DeclareMathOperator{\cf}{cf}

\DeclareMathOperator{\refl}{Refl}
\DeclareMathOperator{\ord}{Ord}

\renewcommand\leq{\leqslant}

\renewcommand\geq{\geqslant}

\renewcommand{\mid}{\mathrel{|}\allowbreak}

\newcommand{\mc}{\mathop{\mathrm{mc}}\nolimits}
\newcommand{\dom}{\mathop{\mathrm{dom}}\nolimits}

\newcommand{\Col}{\mathop{\mathrm{Col}}}

\title[Combinatorics in Singular Solovay Models]{Combinatorics in Higher Solovay Models}
\author[Poveda]{Alejandro Poveda}
\address[Poveda]{Departament de Matemàtiques i Informàtica, Universitat de Barcelona, Barcelona, 08007, Catalonia, Spain. 
}
\address{Fachbereich Mathematik, Universität Hamburg, Bundesstraße 55,
Hamburg, 20146, Germany}
\email{alejandro.poveda@ub.edu}
\urladdr{www.alejandropovedaruzafa.com}

\author[Thei]{Sebastiano Thei}
\address[Thei]{IMPAN, Warsaw, Poland.}
\email{thei91.seba@gmail.com}

\thanks{ The first author acknowledges the support from the Alexander von Humboldt Foundation and from project number PID2023-147428NB-I00 funded by the the Spanish Government. }
\usepackage{graphicx} % Required for inserting images

\begin{document}

\maketitle

\begin{abstract}
We construe the singular-cardinal analogue of the classical Solovay model.
     Starting with large cardinal assumptions in the realm of supercompactness, we show that the our inner model captures a substantial portion of the combinatorics of $L(\mathcal{P}(\kappa))$ that are typically implied by Woodin's axiom $I_0$. %We demonstrate that several $\ad$-like consequences fulfilled by Solovay's $L(\mathbb{R})$ are replicated at the level of singular strong limit cardinals $\kappa$ of countable cofinality. 
    Among other things, we show that in our higher Solovay model there are no $\kappa^+$-sequences of distinct members of $\mathcal{P}(\kappa)$ and that Shelah's approachability property $\AP_\kappa$ fails. We prove that every set in our inner model satisfies a singular analogue of the complete Ramsey property and that the partition relation $\kappa\xrightarrow[]{\mathrm{OD}} (\omega)^\omega_{V_\mu}$ holds for all $\mu<\kappa$. 
\end{abstract}

\begin{comment}
\section{Introduction}

\textcolor{blue}{I think we want to say this: \begin{enumerate}
    \item Similarities between AD and I0
    \item Shi-Trang found more evidences of this similarities
    \item In this paper we want to analyze the possible analogies between Levy collapse and Merimovich. 
    \item In the paper \cite{DPT} some similarities in the context of DST/GDST have already been deduced. 
    \item Here we focus more on the combinatorial properties of $L(\mathcal{P}(\kappa))$.
    \item We also briefly analyze (and leave for future research) the large cardinal structure of $L(\mathcal{P}(\kappa))$.
\end{enumerate}}

\end{comment}
\section{Introduction}\label{sec: intro}

The interaction between determinacy, regularity properties and combinatorics is one
of the central themes of modern set theory. Under the Axiom of Determinacy, the
real line behaves in a remarkably regular way: every set of reals has the Perfect Set
Property, the Baire Property and is Lebesgue measurable \cite{Kan}. Since determinacy is
incompatible with the Axiom of Choice, the canonical inner model in which to study
these phenomena is $L(\mathbb R)$. A celebrated theorem of Solovay shows that,
starting from an inaccessible cardinal, one can force a model of ZFC in which every
set of reals belonging to $L(\mathbb R)$ has the regularity properties normally
associated with determinacy \cite{Sol70}. Thus Solovay's model isolates a large
part of the descriptive-set-theoretic content of $AD^{L(\mathbb R)}$ while starting
from much weaker large-cardinal assumptions; see also \cite{Kec,Kan,KoelWoo} for
background on classical descriptive set theory, determinacy and large cardinals.

\smallskip

There is a parallel story at singular cardinals. Woodin observed that the
rank-into-rank axiom $I_0(\kappa)$, asserting the existence of a non-trivial
elementary embedding
$$
        j:L(V_{\kappa+1})\to L(V_{\kappa+1})
$$
with critical point below $\kappa$, makes $L(V_{\kappa+1})$ behave in many
respects like $L(\mathbb R)$ under $AD^{L(\mathbb R)}$. Here $\kappa$ is
necessarily a singular cardinal of countable cofinality, and $V_{\kappa+1}$ plays
the role of the real line. This analogy has been substantiated by work of Woodin,
Cramer, Shi and others, showing, for instance, that under $I_0(\kappa)$ sets in the
generalized Baire space ${}^{\omega}\kappa$ lying in $L(V_{\kappa+1})$ satisfy
higher analogues of the Perfect Set Property \cite{WoodinPartII,Cramer,Shi}.

\smallskip

The analogy is not limited to descriptive set theory. In their work on $I_0$ and
combinatorics at successors of singular cardinals, Shi and Trang showed that
$L(V_{\kappa+1})$, and equivalently in this context $L(\mathcal P(\kappa))$,
satisfies a striking collection of combinatorial statements at $\kappa^+$:
failures of square-like principles, failures of scales, failures of diamond, failures
of GCH at $\kappa$, and strong reflection-like properties \cite{ShiTrang}. These
results provide further evidence that $I_0(\kappa)$ is the correct higher analogue
of determinacy for the study of the singular cardinal $\kappa$ and its successor.

The purpose of this paper is to investigate how much of this
$I_0$-combinatorics can be recovered from supercompactness alone. More precisely,
we analyze the combinatorics of the singular-cardinal Solovay model
$L(\mathcal P(\kappa)^{V[G]}),$
 where $G$ is generic for Merimovich's supercompact extender-based Prikry forcing
\cite{MerSuper}. This forcing changes the cofinality of $\kappa$ to $\omega$,
makes an inaccessible cardinal $\lambda$ become $(\kappa^+)^{V[G]}$, and supplies
a higher analogue of the classical Solovay construction \cite{DPT}. In the classical case, the
relevant forcing is the L\'evy collapse $\Col(\omega,{<}\lambda)$; in
the present singular setting, its role is played by Merimovich's forcing. One of the
guiding questions of this paper is therefore the following:

\begin{quote}
To what extent does Merimovich's forcing reproduce, at a singular cardinal
$\kappa$, the role played by the L\'evy collapse in Solovay's model?
\end{quote}

This question was addressed from the viewpoint of generalized descriptive set
theory in \cite{DPT}. There it was shown that, in the Merimovich extension, every
subset of ${}^{\omega}\kappa$ belonging to $L(\mathcal P(\kappa)^{V[G]})$ has the
$\kappa$-Perfect Set Property. In the same spirit, one also obtains higher
Baire-property analogues for the relevant Ellentuck--Prikry topology. The present
paper takes the next step; namely, it investigates  combinatorial properties of the higher Solovay model 
$L(\mathcal P(\kappa)^{V[G]})$.

Our first observation is that the $\kappa$-Perfect Set Property alone already has
strong consequences. Under $\zf+\dc_\kappa$, the assertion that every subset of
${}^{\omega}\kappa$ has the $\kappa$-Perfect Set Property rules out several
classical combinatorial objects which exist under choice. In particular, it implies
the non-existence of scales at $\kappa$, the failure of $\sch_\kappa$, the failure
of $\Diamond_{\kappa^+}$, and the non-existence of $\kappa^+$-sequences of
distinct subsets of $\kappa$. Thus the regularity property obtained in the singular
Solovay model already captures several of the consequences previously derived
from $I_0(\kappa)$ in \cite{ShiTrang}.

The second part of the paper uses the specific structure of Merimovich's forcing.
The forcing comes equipped with a rich system of subforcings and projections,
together with the interpolation and constellation lemmas developed in \cite{DPT}. These
tools play the role that factorization and homogeneity play in Solovay's original
argument. They allow us to move between intermediate generic extensions without
changing the full power set of $\kappa$, and hence to control objects definable over
$L(\mathcal P(\kappa)^{V[G]})$. This yields additional combinatorial information
which is not a purely formal consequence of the $\kappa$-Perfect Set Property.

\begin{theorem}[Main theorem]\label{thm: main theorem}
Assume $\gch$. Let $\kappa<\lambda$, where $\lambda$ is inaccessible and $\kappa$
is ${<}\lambda$-supercompact. Let $\mathbb P$ be the corresponding Merimovich's forcing  and let $G\subseteq\mathbb P$ be $V$-generic. Set
$$
        N=L(\mathcal P(\kappa)^{V[G]}).
$$
Then $V[G]$ satisfies that $\kappa$ is a strong limit singular cardinal of
cofinality $\omega$, $(\kappa^+)^{V[G]}=\lambda$, and the inner model $N$ satisfies
$\zf+\dc_\kappa$. Moreover, the following hold in $N$:
\begin{enumerate}
    \item every subset of ${}^{\omega}\kappa$ has the $\kappa$-Perfect Set
    Property;

    \item there are no scales at $\kappa$;

    \item $\sch_\kappa$ fails and $\Diamond_{\kappa^+}$ fails;

    \item there is no $\kappa^+$-sequence consisting of distinct members of
    $\mathcal P(\kappa)$;

    %\item \textcolor{red}{both $ADS_\kappa$ and $NPT(\kappa^+,\aleph_1)$ fail;}

    \item Shelah's approachability property $\AP_\kappa$ fails;

    \item there is a $\kappa$-mad family of size $\kappa$;

    \item for every $\mu<\kappa$, the definable partition relation
    $$
        \kappa \overset{\mathrm{OD}}{\longrightarrow}
        (\omega)^\omega_{V_\mu}
    $$
    holds; equivalently, every ordinal-definable coloring
    $c:[\kappa]^\omega\to V_\mu$ has an infinite homogeneous set.
\end{enumerate}
If in addition $\lambda$ is
measurable in $V$, as witnessed by a normal measure $U$, then in the model
$$
        L(\mathcal P(\kappa)^{V[G]},U)
$$
\begin{enumerate}
    \item[$(\aleph)$]  $\kappa^+$ is measurable.
    \item[$(\beth)$] There is a generic extension $V[G\ast H]$ of $V[G]$ via a $\lambda$-distributive forcing where $\mathrm{Cub}_\lambda\restriction (E^\lambda_{\mathrm{Reg}})^V$ is a $L(\mathcal{P}(\kappa),U)^{V[G\ast H]}$-ultrafilter.
\end{enumerate}
\end{theorem}

The theorem should be compared with the Shi--Trang analysis of
$L(V_{\kappa+1})$ under $I_0(\kappa)$ \cite{ShiTrang}. In their work, many of the
same conclusions are obtained from the existence of the rank-into-rank embedding.
In the present paper the source of the combinatorics is different: the negative
results about scales, $\sch$, diamond and sequences of distinct subsets of $\kappa$
follow from the $\kappa$-Perfect Set Property, while the failure of approachability
and the definable partition relation rely on the fine forcing-theoretic structure of
Merimovich's model. Thus the paper separates two mechanisms which are intertwined
under $I_0(\kappa)$: regularity properties on the one hand, and large-cardinal
strength inside $L(\mathcal P(\kappa))$ on the other.

\smallskip

The model also differs in important ways from the classical Solovay model. For
instance, Mathias \cite{HappyFamilies} asked whether infinite maximal almost disjoint families can exist
in Solovay's model, and this was answered negatively by T\"ornquist
\cite{Tornquist}. In contrast, our singular Solovay model contains
$\kappa$-mad families of size $\kappa$.

%We also investigate fragments of the large-cardinal structure of
%$L(\mathcal P(\kappa)^{V[G]})$ and related definable inner models. If the ground
%model inaccessible $\lambda$ is measurable, then adjoining the ground-model measure
%$U$ to $L(\mathcal P(\kappa)^{V[G]})$ yields a $\dc_\kappa$-model in which
%$\kappa^+$ is measurable.\seba{\tiny this is a repetition no? We already said that in the main thm above I guess} Under stronger assumptions on $\lambda$, similar
%arguments show that $\HOD_{\mathcal P(\kappa)}$ can regard $\kappa^+$ as
%J\'onsson, and, with a Ramsey hypothesis, as Ramsey. These results suggest that
%the large-cardinal structure of $L(\mathcal P(\kappa)^{V[G]})$ is rich, but they
%also leave open a number of natural questions (see \S\ref{sec: open questions}).

\smallskip

 By exploring further parallels with the classical case, we slightly strengthen the large cardinal assumption to deduce a consequence of the $\UBP$ (see Definition \ref{def: kappa baire}), a generalization of the classical Baire Property introduced by Dimonte, Motto Ros and Shi in \cite{DimonteMottoShi} and recently studied in \cite{DimonteIannella}. 
\begin{theorem}
   Suppose that  $\langle \kappa_n\mid n<\omega\rangle$ is an increasing sequence   of ${<}\lambda$-supercompact cardinals where $\lambda$ is an inaccessible cardinal above $\sup_{n<\omega}\kappa_n$. Let $\mathbb{P}^{\mathrm{diag}}$ be the \emph{Diagonal Supercompact Extender Based Prikry} defined in \cite[\S4.2]{DPT}, and let $G\subseteq\mathbb{P}^{\mathrm{diag}}$ be $V$-generic. Then \begin{equation*}
       \text{every subset of $\prod_{n<\omega}\kappa_n$ in $ L(\mathcal{P}(\kappa))^{V[G]}$ has the complete $\vec{\mathcal{U}}$-Ramsey property,}
   \end{equation*} where $\vec{\mathcal{U}}:=\langle U_n\mid n<\omega\rangle$ is the sequence of normal measures on $\kappa_n$  induced by some (fixed) elementary embedding $j_n\colon V\to M_n$ witnessing the ${<}\lambda$-supercompactness of $\kappa_n$. Moreover, in $L(\mathcal{P}(\kappa))^{V[G]}$ properties (1)--(7) of Theorem \ref{thm: main theorem} hold.
\end{theorem}

The paper is organized as follows. In Section~\ref{sec: Merimovich forcing} we collect the relevant
preliminaries about Merimovich's forcing
\cite{MerSuper} and Generalized Descriptive Set Theory, and prove the combinatorial consequences of the $\kappa$-Perfect Set Property ($\S\ref{sec: combinatorial consequences of PSP}$). Section~\ref{sec: Combinatorics in Merimovich} turns to arguments specific to Merimovich's forcing. We prove that
$\AP_\kappa$ fails in $L(\mathcal P(\kappa)^{V[G]})$, construct $\kappa$-mad
families there, and establish the ordinal-definable partition relation
$\kappa \overset{\mathrm{OD}}{\longrightarrow}
        (\omega)^\omega_{V_\mu}$ for all $\mu<\kappa$.
We also prove a higher Ellentuck--Prikry Ramsey theorem in the diagonal Merimovich setting \cite{DPT},
analyze the J\'onsson and Ramsey behavior of $\HOD_{\mathcal P(\kappa)}$, and show
that adding a ground-model normal measure $U$ yields a model
$L(\mathcal P(\kappa)^{V[G]},U)$ in which $\kappa^+$ is measurable. Finally, Section~\ref{sec: open questions} collects a few open questions.

\section{Preliminaries}\label{sec: Merimovich forcing}

\subsection{Review on generalized descriptive set theory}\label{sec: gdst}

%\subsection{Dependent Choice}

Classical Descriptive Set Theory \cite{Kec} studies the properties of definable subsets of Polish spaces, with the real line \(\mathbb{R}\) and the Baire space ${}^\omega\omega$ serving as paradigmatic examples. A number of results in the area show that simply definable sets  (e.g., \emph{Borel} or \emph{analytic} sets) do possess a rich canonical structure theory. 
\emph{Generalized Descriptive Set Theory} ($\axiomfont{GDST}$) is preoccupied with the study of definable objects beyond the continuum. Specifically, $\axiomfont{GDST}$'s primary interest is the  study of definable sets in higher function spaces like $^\kappa 2$ and $^{\cf(\kappa)} \kappa$ (see \cite{Friedman}, \cite{DimMot} and \cite[\S2]{DPT}). It turns out that when $\kappa$ is an infinite cardinal with $\cf(\kappa)=\omega$, there are a number of results that parallel the findings of classical descriptive set theory. For instance, the following is a higher singular version of Solovay's fundamental theorem \cite{solovay1970model}:
\begin{theorem}[\cite{DPT}]\label{PSP in Solovay}
    Let $G$ be $\mathbb{P}$-generic over $V$, where $\mathbb{P}$ is the Merimovich forcing defined in $\S\ref{subsec: Meromovich}$. Then every subset of ${}^\omega\kappa$ in $L(\mathcal{P}(\kappa)^{V[G]})$ has the $\kappa$-$\psp$.
\end{theorem}
 Recall that a set $A\s {}^\omega \kappa$ has the $\kappa$\emph{-Perfect Set Property} (briefly, $\kappa$\emph{-\textsf{PSP}}) if either $|A|\leq\kappa$, or there exists a topological embedding from $^\kappa 2$ to $A$, closed-in-${}^\omega\kappa$. If $\kappa$ is a strong limit singular cardinal of countable cofinality, then a set $A\s {}^\omega \kappa$ has the $\kappa$-$\psp$ if and only if either $|A|\leq\kappa$ or there is a continuous injection $\iota:{}^\kappa 2\rightarrow{}^\omega\kappa$ with $\iota``({}^\kappa2)\s A$ (see \cite[Corollary 7.1.3]{DimMot}). Notice that letting $\kappa=\aleph_0$ one recovers the usual $\psp$. In future sections we  explore the consequences of the $\kappa$-$\psp$ upon the combinatorics of $L(\mathcal{P}(\kappa)^{V[G]})$.

\smallskip

The base axiomatic theory to develop $\axiomfont{GDST}$ is $\zf+\dc_\kappa$.  Given an infinite cardinal $\kappa$ the axiom of \emph{Dependent Choice} at $\kappa$ ($\dc_\kappa$) is the following statement: \begin{equation}
    \tag{$\dc_\kappa$}\forall X \, \forall F \colon {}^{<\kappa}X \to \mathcal{P}(X) \setminus \{\emptyset\} \; \exists g \colon \kappa \to X \; \forall \gamma < \kappa \; g(\gamma) \in F(g \restriction \gamma). 
\end{equation}
As anticipated, our models of interest are of the form $L(\mathcal{P}(\kappa))$. Accordingly, we shall argue that they always satisfy $\dc_\kappa$. (Notice that they  satisfy $\zf$, as well.)
\begin{lemma}\label{coding into subsets of kappa}
    $H_{\kappa^+}\s L(\mathcal{P}(\kappa))$, for all cardinals $\kappa$.
\end{lemma}
\begin{proof}
    If $x\in H_{\kappa^+}$, then the transitive closure of $\{x\}$  has cardinality $\eta$ for some cardinal $\eta\leq\kappa$. Let $f\colon \eta\rightarrow\mathrm{tcl}(\{x\})$ be a bijection, and define $\triangleleft\subseteq\eta\times\eta$ as $\text{$\xi\triangleleft\delta:\Longleftrightarrow f(\xi)\in f(\delta)$.}$ Note that $\triangleleft$ can be coded as a subset of $\eta\leq\kappa$ (e.g., using G{\"o}del's pairing function) and therefore $(\eta,\triangleleft)\in L(\mathcal{P}(\kappa)).$ If we compute  the transitive collapse of $(\eta,\triangleleft)$ (in $L(\mathcal{P}(\kappa))$) it has to take the form $(\mathrm{tcl}(\{x\}),\in)$ and thus,  by transitivity, $x\in L(\mathcal{P}(\kappa))$.
\end{proof}
The above lemma will help us show  that $\dc_\kappa$ is downward absolute from $V$ to $L(\mathcal{P}(\kappa))$. The next argument is inspired by arguments presented in \cite{DimonteRankIntoRank}:
\begin{prop}[$\zf$]\label{prop: DC}
If $\dc_\kappa$ holds then $L(\mathcal{P}(\kappa))\models \dc_\kappa$.
\end{prop}
\begin{proof}
    We have to prove that
$$ \forall X \, \forall F \colon {}^{<\kappa}X \to \mathcal{P}(X) \setminus \{\emptyset\} \; \exists g \colon \kappa \to X \; \forall \gamma < \kappa \; g(\gamma) \in F(g \restriction \gamma) $$ holds in $L(\mathcal{P}(\kappa))$. First we verify $\dc_\kappa(\mathcal{P}(\kappa))$; i.e., $\dc_\kappa$ for $X=\mathcal{P}(\kappa)$. To see this recall that $\dc_\kappa$ holds in $V$ and so for all $F \colon {}^{<\kappa}\mathcal{P}(\kappa) \to \mathcal{P}(\mathcal{P}(\kappa)) \setminus \{\emptyset\}$ there is $ g \colon \kappa \to \mathcal{P}(\kappa) $ with $\forall \gamma < \kappa \; g(\gamma) \in F(g \restriction \gamma)$. But $g\in H_{\kappa^+}$ and so Lemma \ref{coding into subsets of kappa} yields $g\in L(\mathcal{P}(\kappa))$. Now we prove $\dc_\kappa(\alpha \times \mathcal{P}(\kappa))$ for every ordinal $\alpha$. Accordingly pick $F \colon {}^{<\kappa}(\alpha\times\mathcal{P}(\kappa)) \to \mathcal{P}(\alpha\times\mathcal{P}(\kappa)) \setminus \{\emptyset\}$, and for each $s \in {}^{<\kappa}(\alpha \times \mathcal{P}(\kappa))$, stipulate $m(s):=\min\{\gamma\in\ord\mid\exists x\in \mathcal{P}(\kappa)\; (\gamma, x)\in F(s)\}$, and let $\pi_2 : {}^{<\kappa}(\alpha \times \mathcal{P}(\kappa)) \to {}^{<\kappa}\mathcal{P}(\kappa)$ be the projection mapping $\langle (\gamma_\xi,x_\xi) \mid \xi < \nu \rangle$ to $\langle x_\xi \mid \xi < \nu \rangle$. We construct by induction a function $c:{}^{<\kappa}\mathcal{P}(\kappa)\to {}^{<\kappa}(\alpha \times \mathcal{P}(\kappa))$ such that $\pi_2(c(t)) \subseteq t$ for all $t\in {}^{<\kappa}\mathcal{P}(\kappa)$. If $t=\emptyset$, stipulate $c(t):=\emptyset$. If $t=\langle x_\xi \mid \xi < \nu \rangle \in {}^{<\kappa}\mathcal{P}(\kappa)$ and for each $\xi<\nu$ the image $c(t\restriction\xi)$ is defined, define $\bar{\nu}:=\min\{\eta<\nu\mid\forall\gamma<\alpha\; (\gamma,x_\eta)\notin F(c(t\restriction\eta))\}$ and $\gamma_\xi:=\min\{\gamma<\alpha\mid (\gamma,x_\xi)\in F(c(t\restriction\xi))\}$ for each $\xi<\bar{\nu}$. Stipulate $c(t):=\langle (\gamma_\xi, x_\xi) \mid \xi < \bar{\nu} \rangle$, where $\gamma_\xi := \min \{\gamma \mid (\gamma, x_\xi) \in F(c(t \restriction \xi))\}.$ Let $G:{}^{<\kappa}\mathcal{P}(\kappa) \to \mathcal{P}(\mathcal{P}(\kappa)) \setminus \{\emptyset\}$ defined as
$$ G(t) := \{x \in \mathcal{P}(\kappa) \mid (m(c(t)), x) \in F(c(t))\}.$$ By $\dc_\kappa(\mathcal{P}(\kappa))$ there is $g : \kappa \to \mathcal{P}(\kappa)$ in $L(\mathcal{P}(\kappa))$ such that for all $\beta < \kappa$, $g(\beta) \in G(g \restriction \beta)$. Now let $f : \kappa \to \alpha \times \mathcal{P}(\kappa)$ be defined by recursion as $f(\beta) := (m(f \restriction \beta), g(\beta))$. 
\begin{claim}
    $f(\beta) \in F(f \restriction \beta)$ for every $\beta < \kappa$.
\end{claim}
\begin{proof}[Proof of claim]
    We prove by induction that $f \restriction \beta = c(g \restriction \beta)$. Suppose that for every $\xi < \beta$, $f \restriction \xi = c(g \restriction \xi)$. By definition $c(g \restriction \beta) = \langle (\gamma_\xi, g(\xi)) \mid \xi < \bar{\beta} \rangle$, where $\gamma_\xi := \min\{\gamma \mid (\gamma, g(\xi)) \in F(c(g \restriction \xi))\}$.
So $(\gamma_\xi, g(\xi)) \in F(c(g \restriction \xi))$. But $g(\xi) \in G(g \restriction \xi)$, so by definition of $G$, it holds that $(m(c(g \restriction \xi)), g(\xi)) \in F(c(g \restriction \xi))$.
Therefore, $\gamma_\xi = m(c(g \restriction \xi))$ and
$$ f(\xi) = (m(f \restriction \xi), g(\xi)) = (m(c(g \restriction \xi)), g(\xi)) = (\gamma_\xi, g(\xi)) = c(g \restriction \beta)(\xi). $$
So $f \restriction \beta = c(g \restriction \beta)$ and, since for every $\xi$, $(\gamma_\xi, g(\xi)) \in F(c(g \restriction \xi))$, we conclude $f(\beta) \in F(f \restriction \beta)$.
\end{proof}

Finally, let $X \in L(\mathcal{P}(\kappa))$. Fix a definable surjection $\Phi$ from $\ord\times\mathcal{P}(\kappa)$ onto $L(\mathcal{P}(\kappa))$, and let $\alpha$ be such that $\Phi``(\alpha \times \mathcal{P}(\kappa)) \supseteq X$. Let $F : {}^{<\kappa}X \to \mathcal{P}(X) \setminus \{\emptyset\}$. For $t = \langle (\gamma_\xi, x_\xi) \mid \xi < \nu \rangle \in {}^{<\kappa}(\alpha \times \mathcal{P}(\kappa))$, stipulate $c(t) := \langle \Phi(\gamma_\xi, x_\xi) \mid \xi < \bar{\nu} \rangle \in {}^{<\kappa}X$, where $\bar{\nu}$ is the minimum such that $\Phi(\gamma_{\bar{\nu}}, x_{\bar{\nu}}) \notin X$. Define $G \colon {}^{<\kappa}(\alpha \times \mathcal{P}(\kappa)) \to \mathcal{P}(\alpha \times \mathcal{P}(\kappa)) \setminus \{\emptyset\}$
as $$ G(t) := \{(\gamma, x) \mid \Phi(\gamma, x) \in F(c(t))\}.$$ By $\dc_\kappa(\alpha \times \mathcal{P}(\kappa))$ there is $g : \kappa \to \alpha \times \mathcal{P}(\kappa)$ in $L(\mathcal{P}(\kappa))$ such that $g(\beta) \in G(g \restriction \beta)$ for every $\beta < \kappa$. Then $f = \Phi \circ g$ is as we wanted, because for every $\beta < \kappa$, $\Phi(g(\beta)) \in F(c(g \restriction \beta))$, and as above we can prove that $c(g \restriction \beta) = f \restriction \beta$.

So $\dc_\kappa(X)$ holds in $L(\mathcal{P}(\kappa))$, whenever $X \in L(\mathcal{P}(\kappa))$. Hence, $L(\mathcal{P}(\kappa))\models\dc_\kappa$.
\end{proof}
%The theory $\zf+\dc_\kappa$  will grant that the classical  proofs go smoothly in the generalized context.
%\begin{definition}[{\cite{DimonteMotto}}]
 %   For any cardinal $\kappa$, a topological space $\mathcal{X}$ is called \emph{$\kappa$-Polish} in case it is homeomorphic to a completely metrizable space with weight $\kappa$.
%\end{definition}
%Given a non-empty set $X$ and an infinite cardinal $\kappa$, the set ${}^\kappa X$ of all functions from $\kappa$ to $X$ will be equipped with the so-called \emph{bounded topology}; namely, the topology whose basic open neighborhoods are $N_{\xi, s}:=\{x\in{}^\kappa X\mid x\restriction\xi=s\}$, where $\xi<\kappa$ and $s\in{}^\xi X.$ In case $\kappa=\omega$ this coincides with the product  of the discrete topologies in $X.$ If $\kappa$ is a strong limit cardinal of countable cofinality, then the \emph{Generalized Cantor Space} ${}^\kappa 2$ and the \emph{Generalized Baire Space} ${}^\omega\kappa$ are examples of $\kappa$-Polish spaces. 

Now we briefly discuss the higher analogue of the \emph{Baire property}. A topological space $\mathcal{X}$ is a  \emph{Baire space} if every non-empty open subset of $\mathcal{X}$ is not meager. A set  $A\s \mathcal{X}$ has the \emph{Baire Property} ($\bp$) if there is an open set $O\s \mathcal{X}$ such that $A\triangle O$ is meager. Dimonte--Motto Ros--Shi \cite{DimonteMottoShi} define the notion of a \emph{$\kappa$-Baire space} and $\kappa$-\emph{Baire Property} ($\kappa$-$\bp$) by replacing ``meager'' by $``\kappa$-meager'' (i.e., a union of $\kappa$-many nowhere dense sets) in the classical definition.

\smallskip

The notion of a $\kappa$-Baire space is meaningful for general topological spaces but it does not fit well with the canonical $\kappa$-Polish topologies. For instance, suppose that $\langle\kappa_n\mid n<\omega\rangle$ is an increasing sequence of regular cardinals, set $\kappa:=\sup_{n<\omega}\kappa_n$, and consider the product space $\prod_{n<\omega}\kappa_n$. As it turns out, this is \textbf{not} a $\kappa$-Baire space in its $\kappa$-Polish topology because it is the union of $\omega_1$-many nowhere dense sets \cite{DimonteMottoShi}.\footnote{For each $\alpha<\omega_1$, set $U_\alpha=\{x\in \prod_{n<\omega}\kappa_n\mid\exists n<\omega\, x(n)=\alpha\}$ and note that $\prod_{n<\omega}\kappa_n=\bigcup_{\alpha<\omega_1}(\prod_{n<\omega}\kappa_n\setminus U_\alpha)$. Since the $U_\alpha$'s are dense open, the conclusion follows.} To circumvent this issue another route is outlined in \cite{DimonteMottoShi}. Namely, the authors consider $\prod_{n<\omega}\kappa_n$ with a topology auxiliary to its natural product topology; namely, the $\vec{\mathcal{U}}$-\emph{Ellentuck-Prikry} topology ($\UEP$). %While we will insist on $C(\Sigma)$ retaining its $\kappa$-Polish topology (in fact, $C(\Sigma)$ is not $\kappa$-Polish in the $\UEP$ topology) it is in the $\UEP$ topology where the $\kappa$-$\bp$ for subsets  $A\s C(\Sigma)$ will be formulated. This is because $C(\Sigma)$ is a $\kappa$-Baire space with respect to the auxiliary $\UEP$ topology \cite{DimonteMottoShi}.

The choice of the $\UEP$ topology is inspired by the fact that in the classical setting the product topology of the Baire space ${}^\omega\omega$ is homeomorphic to the topology of maximal filters on Cohen forcing.\footnote{This is the topology generated by the open sets  $\mathcal{N}_p=\{F\in \beta\omega\mid \supp(p)\in F\}$ where $\supp(p)$ is the support of a condition $p\colon \omega\rightarrow 2$ in Cohen forcing $\mathrm{Add}(\omega,1).$}
The analogy in the singular case is provided by Magidor's \emph{Diagonal Prikry forcing} (Definition \ref{def: diagonal prikry} below).

\smallskip

Hereafter we assume that $\vec{\mathcal{U}}=\langle \mathcal{U}_n\mid n<\omega\rangle$ is a sequence of normal (non-principal) ultrafilters on the members $\kappa_n$ of an increasing sequence $\Sigma=\langle \kappa_n\mid n<\omega\rangle$. We shall set $\kappa:=\sup(\Sigma)$. Recall $C(\Sigma)$ denotes the $\kappa$-Polish space $\prod_{n<\omega}\kappa_n.$

\begin{definition}[Diagonal Prikry forcing (Magidor)]\label{def: diagonal prikry}
A condition in the \emph{Diagonal Prikry forcing with $\vec{\mathcal{U}}$} (in symbols, $\mathbb{P}(\vec{\mathcal{U}})$) is a sequence $$p=\langle \alpha^p_0,\dots, \alpha^p_{\ell(p)-1}, A^p_{\ell(p)}, A^p_{\ell(p)+1},\dots \rangle$$ where
$s^p=\langle \alpha^p_0,\dots, \alpha^p_{\ell(p)-1}\rangle\in \prod_{n<\ell(p)}\kappa_n$ is strictly increasing,  $A^p_n$ belongs to $\mathcal{U}_n$ for $n\geq \ell(p)$ and every $\beta\in A_{n+1}^p$ is bigger than $\kappa_n$. 

Given conditions $p,q\in \mathbb{P}(\vec{\mathcal{U}})$ write $p\leq q$ if $s^p$ end-extends $s^q$ (in symbols,  $s^p\sq s^q$), $s^p(n)\in A^q_n$ for $n\in [\ell(q),\ell(p))$ and $A^p_n\s A^q_n$ for all $n\geq \ell(p).$
\end{definition}
\begin{definition}[The $\UEP$ topology]\label{def: ellentuck prikry}\hfill
   \begin{enumerate}
    \item For each $x\in \prod_{n<\omega}\kappa_n$ we consider the  filter $\mathcal{F}_x\s \mathbb{P}({\vec{\mathcal{U}}})$ defined as
    $$\mathcal{F}_x:=\{p\in \mathbb{P}({\vec{\mathcal{U}}})\mid s^p\sqsubseteq x\,\wedge\,\forall n\geq \ell(p)\, (x(n)\in A^p_n)\}$$
and we say that \emph{$x$ is $\mathbb{P}(\vec{\mathcal{U}})$-generic}  if $\mathcal{F}_x$ is a $\mathbb{P}(\vec{\mathcal{U}})$-generic filter.  
       \item For each condition $p\in \mathbb{P}({\vec{\mathcal{U}}})$, define $$\textstyle \mathcal{N}_p:=\{x\in\prod_{n<\omega}\kappa_n\mid p\in\mathcal{F}_x\}.$$ 
       \item  The \emph{$\vec{\mathcal{U}}$-Ellentuck-Prikry}  ($\UEP$) topology $\mathcal{T}_{\UEP}$ is the topology which has  $\{\mathcal{N}_p\mid p\in\mathbb{P}({\vec{\mathcal{U}}})\}$ as basic open sets. 
   \end{enumerate}
\end{definition}

\begin{definition}[$\vec{\mathcal{U}}$-Baire Property]\label{def: kappa baire}
   A set $A\subseteq C(\Sigma)$ has the \emph{$\vec{\mathcal{U}}$-Baire Property} ($\UBP$) if it has the $\kappa$-$\bp$ as a subset of the topological space $(C(\Sigma), \mathcal{T}_{\UEP})$; namely, if  there is a $\UEP$-open set $O\subseteq C(\Sigma)$ such that $A\Delta O$ is $\kappa$-meager in  $\mathcal{T}_{\UEP}$. 
\end{definition}

\begin{fact}\hfill
\begin{enumerate}
      \item (\cite{DimonteMottoShi})  $(C(\Sigma), \mathcal{T}_{\UEP})$ is a $\kappa$-Baire space.  
    \item (\cite{DimonteMottoShi})  All the $\kappa$-analytic subsets of $\prod_{n<\omega}\kappa_n$ have the $\vec{\mathcal{U}}$-$\bp$.
    \item (\cite{DimonteIannella})  There exists a subset of $\prod_{n<\omega}\kappa_n$ without the $\vec{\mathcal{U}}$-$\bp$.
\end{enumerate}
\end{fact}
In \cite[\S4.2]{DPT}, Dimonte--Poveda--Thei constructed a model of the form $L(\mathcal{P}(\kappa))$, with $\kappa=\sup_{n<\omega}\kappa_n$, where every set $A\s \prod_{n<\omega}\kappa_n$ has the $\UBP$.

\smallskip

A related  property is the \emph{Completely $\vec{\mathcal{U}}$-Ramsey property}:
\begin{definition}[Complete $\vec{\mathcal{U}}$-Ramsey property]
    A set $A \s \prod_{n<\omega\omega}\kappa_n$ has the complete $\vec{\mathcal{U}}$-Ramsey property if for every $p\in \mathbb{P}(\vec{\mathcal{U}})$ there is $q\leq^* p$ such that  $$\text{$\mathcal{N}_q\s A$ or $\mathcal{N}_q\cap A=\emptyset$.}$$
\end{definition}
The above is the $\mathcal{T}_{\UEP}$-version of the classical \emph{Complete Ramsey Property} in  $([\mathbb{N}]^\infty,\mathcal{T}_{\mathrm{EP}})$, where $\mathcal{T}_{\mathrm{EP}}$ denotes the Ellentuck--Prikry topology \cite{Ellentuck}. 

Mathias in \cite{MathiasRamsey} and Silver in \cite{SilverAnalytic} have showed that every analytic set has the Complete Ramsey property. In Section~\ref{sec: the complete Ramsey} we will extend this to the realm of singular cardinals by showing that in the model of \cite[\S4.2]{DPT} (i.e., in the $L(\mathcal{P}(\kappa))$ of the generic extension by the Diagonal Supercompact Extender Based Prikry) every subset of $\prod_{n<\omega}\kappa_n$ has the complete $\vec{\mathcal{U}}$-Ramsey property.

\subsection{Review of Merimovich's forcing}\label{subsec: Meromovich}
In this section we briefly review  Merimovich's forcing from \cite{MerSuper} utilizing  the exposition provided in \cite[\S4]{DPT}.\footnote{The main difference with respect to \cite{DPT} is that here we use trees in place of measure one sets in the definition of Merimovich's forcing. This is useful when showing that there are projections between $\mathbb{P}$ and its various subforcings. } 

In what follows we assume the GCH  and   
that $j\colon V\rightarrow M$ is an elementary embedding  with $\crit(j)=\kappa$ and $M^{<\lambda}\s M$, where $\lambda$ is an inaccessible cardinal above $\kappa$.
\begin{definition}[Domains]
    A \emph{domain} is a set $d\in [\lambda\setminus \kappa]^{<\lambda}$ with $\kappa=\min(d)$.  The collection of all domains will be denoted by $\mathcal{D}^\ast:=\mathcal{D}(\kappa,\lambda)$.
\end{definition}

%Given  $d\in \mathcal{D}^\ast$ there is a $\kappa$-complete ultrafilter $E(d)$ attached to $d$. This ultrafilter does not concentrate on a set of ordinals, but rather on the set of \emph{$d$-objects}, which is introduced in the next definition:
\begin{definition}[$d$-object]
    Let $d\in\mathcal{D}^\ast$. A function $\nu\colon \dom(\nu)\rightarrow\kappa$ is called a \emph{$d$-object} if it fulfills the following requirements; namely,
    \begin{enumerate}
        \item $\kappa\in\dom(\nu)\s d$ and $\nu(\kappa)$ is an inaccessible cardinal;
        \item $\nu(\alpha)<\nu(\beta)$ for each $\alpha<\beta$ in $\dom(\nu)$;
      
    \end{enumerate}
    The set of $d$-objects will be denoted by $\ob(d)$. Given $\nu,\mu\in \ob(d)$ we write $\nu\prec \mu$ if $\dom(\nu)\s \dom(\mu)$ and $\nu(\alpha)<\mu(\kappa)$ for all $\alpha\in\dom(\nu)$. 
\end{definition}

The definition of a $d$-object embodies the main features of $\mc(d)$ (the \emph{maximal coordinate} of $d$) in the  $M$-side of the master embedding $j$:
$$\mc(d):=\{\langle j(\alpha),\alpha\rangle\mid \alpha\in d\}.$$

\begin{definition}[Ultrafilters on $\ob(d)$]
  Given $d\in \mathcal{D}^\ast$ define $$E(d):=\{X\s \ob(d)\mid \mc(d)\in j(X)\}.$$ 
\end{definition}
\begin{remark}
    A few data points about $E(d)$. First, $E(d)$ is a $\kappa$-complete ultrafilter, yet not necessarily normal. Second, given domains $d\s e$ there is a natural projection between $\ob(e)$ and $\ob(d)$ given by $\nu\mapsto \nu\restriction d$. This in turn induces a \emph{Rudin-Keisler projection} between $E(e)$ and $E(d)$ which we  denote by $\pi_{e,d}$ or $\restriction d$ (if $e$ is clear from the context). 
\end{remark}

\begin{definition}
	Let $d\in\mathcal{D}^\ast$. A tree $T\s\ob(d)^{<\omega}$ is called an $E(d)$-tree if it consists of $\prec$-increasing sequences of $d$-objects and for each $\vec\nu\in T$, $$\mathrm{Succ}_T(\vec\nu):=\{\mu\in\ob(d)\mid \vec\nu{}^\smallfrown\langle \mu\rangle\in T\}\in E(d).$$
	Given an $E(d)$-tree $T$ and a sequence of $d$-objects $\vec\nu\in T$, denote
	$$T_{\vec\nu}:=\{\vec\eta\in\ob(d)^{<\omega}\mid \vec\nu{}^\smallfrown\vec\eta\in T\}.$$
\end{definition}

\begin{definition}[\cite{MerSuper}]
   The poset $\mathbb{P}$  consists of pairs  
    $p=\langle f,T\rangle$
    where:
    \begin{enumerate}

        \item $f\colon \dom(f)\rightarrow{}^{<\omega}\kappa$ is a function with $\dom(f)\in\mathcal{D}^\ast$ and
        $$f(\alpha)=\langle f_0(\alpha),\dots, f_{|f(\alpha)|-1}(\alpha)\rangle$$
        is increasing, for all $\alpha\in\dom(f)$.
        \item $T$ is an $E(\dom(f))$-tree. Furthermore,  for each $\langle \nu\rangle\in T$
        \begin{itemize}
            \item  $\nu(\kappa)>\sup(\text{ran}(c_{n-1}))$;
            \item and $\nu(\kappa)>\max(f(\alpha))$ for all $\alpha\in\dom(\nu)$.
        \end{itemize}

    \end{enumerate}
    Given $p\in\mathbb{P}$ its \emph{length} (denoted $\ell(p)$) is the integer $n$.

    \smallskip

    Following Merimovich \cite{MerSuper} we denote by $\mathbb{P}^*$ the poset consisting  of the functions $f$ from item (1) ordered by $\supseteq$-inclusion.
\end{definition}

\begin{definition}[Pure extensions]
    Given 
      $p=\langle f^p,T^p\rangle$ and
        $q=\langle f^q,T^q\rangle$ in $\mathbb{P}$
        we write $q\leq^* p$ whenever   $f^p\s f^q$ and $T^q\restriction \dom(f^p)\s T^p$, where  $$T^q\restriction\dom(f^p):=\{\langle \nu_0\restriction\dom(f^p), \dots, \nu_m\restriction\dom(f^p)\rangle \mid \langle \nu_0,\dots,\nu_m\rangle \in T^q\}.$$
\end{definition}

\begin{definition}[One point extensions]
Given a condition 
 $p=\langle f,T\rangle$
and $\langle \nu\rangle\in T$, the \emph{one-point extension of $p$ by $\langle \nu\rangle$} (in symbols, $p\cat\langle\nu\rangle$) is
$$\langle f_{\langle \nu\rangle},T_{\langle \nu\rangle}\rangle$$
where 
$$f_{\langle\nu\rangle}:=\begin{cases}
   f(\alpha){}^\smallfrown \langle\nu(\alpha)\rangle, & \text{if $\alpha\in\dom(\nu)$,}\\
   f(\alpha), & \text{otherwise,}
\end{cases}
$$
and $T_{\langle\nu\rangle}:=\{\vec\eta\in \ob(\dom(f))^{<\omega}\mid \langle\nu\rangle^\smallfrown \vec\eta\in T\}.$

Given a $\prec$-increasing sequence of objects $\vec\nu\in A^{<\omega}$ one defines $p\cat\vec\nu$ by recursion on the length of $|\vec\nu|$ setting as a base case $p\cat\varnothing:=p$.
\end{definition}
\begin{fact}
    For each $p\in\mathbb{P}$ and $\vec\nu$ a $\prec$-increasing sequence of objects in the tree of $p$, $p\cat\vec\nu$ is a legitimate condition in $\mathbb{P}$. \qed
\end{fact}
\begin{definition}[The main ordering]
    Given conditions $p,q\in\mathbb{P}$ we write $q\leq p$ if there is a $\prec$-increasing sequence of objects $\vec\nu$ in the tree of $p$ such that $q\leq^*p\cat\vec\nu$.
\end{definition}

\begin{theorem}[{Essentially \cite{MerSuper}, see also \cite[Lemma~4.5]{DPT}}]\label{lemma: PE SigmaPrikry}\hfill

 $\mathbb{P}$ is a $\Sigma$-Prikry poset taking $\Sigma:=\langle \kappa\mid n<\omega\rangle$. \qed
\end{theorem}

Next we describe the various natural subforcings of $\mathbb{P}:$

\begin{definition}
    For each $d\in\mathcal{D}^\ast$ denote by $\mathbb{P}_d$ the subposet of $\mathbb{P}$ whose universe is
    $\{p\in\mathbb{P}\mid \dom(f^p)\s d\}.$ 
\end{definition}

\begin{remark}
    Since $\{\kappa\}$ is a domain, $\mathbb{P}_{\{\kappa\}}$ is well-defined. A moment of reflection makes clear that $\mathbb{P}_{\{\kappa\}}$ is essentially the usual Tree Prikry forcing \cite{Gitik-handbook}.
\end{remark}
The forthcoming Lemma~\ref{lemma: quasi projections} can be proved as in \cite[Lemma~4.9]{DPT}. The main difference compared to \cite{DPT} is that here we obtain a commutative system of \textbf{projections} rather than just mere weak projections. This is the reason why in  \cite{DPT} we had to pass to the Boolean completions of the forcings $\mathbb{P}_e$. It is precisely the fact that we are considering trees (and not just measure one sets, as in \cite{DPT}) which makes each of the $\pi_{e,d}$'s a projection. The key observation is the following: Let $p\in\mathbb{P}_e$  and $q\leq \pi_{e,d}(p)$. Let $\vec\nu$ be a ($\prec$-increasing) sequence $\vec\nu\in (T^p\restriction d)^{<\omega}$ such that $q\leq^*\pi_{e,d}(p)\cat\vec\nu$. By definition, $T^p\restriction d:=\{\langle \mu_0\restriction d, \dots, \mu_n\restriction d\rangle\mid \langle \mu_0, \dots, \mu_n\rangle\in T^p\}.$
Members of $T^p$ are $\prec$-increasing sequences, hence there is a $\prec$-increasing sequence $\langle \tau_0,\dots, \tau_{n}\rangle\in T^p$  whose $d$-projection is $\vec\nu$. Thus, $\langle \tau_0,\dots, \tau_{n}\rangle$  is ``addable'' to $p$; in particular, $p\cat\langle \tau_0,\dots\tau_n\rangle$ is a well-defined condition. Bearing this in mind, the argument in \cite[Lemma~4.9]{DPT} yields the following:
\begin{lemma}\label{lemma: quasi projections}
There is a commutative system of  projections 
$$\mathcal{P}=\langle \pi_{e,d}\colon \mathbb{P}_e\rightarrow\mathbb{P}_d\mid d\s e\,\wedge\, e,d\in\mathcal{D}^\ast\rangle$$
given by $\pi_{e,d}\colon p\mapsto \langle  f^p\restriction d, T^p\restriction d\rangle$. Also,   $|\mathrm{tcl}(\mathbb{P}_d)|<\lambda$ for all $d\in \mathcal{D}^\ast.$

In particular, if $G\s\mathbb{P}_e$ is generic then $\pi_{e,d}``G$ induces a $\mathbb{P}_d$-generic. \qed
\end{lemma}

\begin{lemma}[Cardinal structure, {\cite{SupercompatRadinExtender}}]\label{lemma: cardinal structure in PE}\hfill
\begin{enumerate}
    \item $\mathbb{P}$ is $\lambda^+$-cc and  preserves both $\kappa$ and $\lambda$.
    \item $\one\forces_{\mathbb{P}}``(\kappa^+)^{V[\dot{G}]}=\lambda\,\wedge\,\cf(\kappa)^{V[\dot{G}]}=\omega$''.\qed
\end{enumerate}
\end{lemma}
\begin{lemma}[Capturing, essentially {\cite[Lemma~4.10]{DPT}}]\label{lemma: capturing subsets}
    Let $G$ a $\mathbb{P}$-generic filter. For each $a\in\mathcal{P}(\kappa)^{V[G]}$ there is  $d\in \mathcal{D}^\ast$ such that $a\in \mathcal{P}(\kappa)^{V[\pi_d``G]}$.

  {Moreover, for every set $a\in \mathcal{P}(\lambda)^{V[G]}$ with $|a|^{V[G]}<\lambda$ there is $d\in \mathcal{D}^*$ such that $d\in V[\pi_d``G].$}
%\footnote{This is sufficient to establish the \emph{$\kappa$-capturing} property displayed in Definition~\ref{def: nice system}: Let $d\in\mathcal{D}^*$ be as in the lemma. Given $d_0\in\mathcal{D}^*$, $d_0\sle e:=d\cup d_0$ and $x\in \mathcal{P}(\kappa)^{V[g_{e}]}$.}
    \qed
\end{lemma}

For later use we also prove the following technical lemma:

\begin{lemma}[Homogeneity]\label{lem: hom of the poset}
Let $\varphi(x,x_0,\dots,x_n)$ be a first order formula, and let $\dot{b},\dot{a}_0,\dots,\dot{a}_n$ be $\mathbb{P}_e$-names for some domain $e\in\mathcal{D}^\ast$. Then, for each $p\in\mathbb{P}$ with $e\s \dom(f^p)$ the following conditions are equivalent:
\begin{enumerate}
    \item $ p \Vdash_{\mathbb{P}}\varphi(\dot{b},\dot{a}_0,\dots,\dot{a}_n).$
    \item   $ \pi_e(p) \Vdash_{\mathbb{P}_e}\varphi(\dot{b},\dot{a}_0,\dots,\dot{a}_n).$
\end{enumerate}
\end{lemma}
\begin{proof}

(2) $\Rightarrow$ (1) is immediate. For (1) $\Rightarrow$ (2) we perform a density argument. Let $r\leq \pi_e(p)$ be a condition. We will use the following fact:

\begin{claim}
    There are  $p^*_0\leq p$ and $p^*_1\leq^* r$ and an isomorphism $$\gamma\colon \mathbb{P}/p^*_0\rightarrow \mathbb{P}/p^*_1$$ such that $\pi_e``G=(\pi_e\circ \gamma)``G$
    for every generic filter $G\s \mathbb{P}/p^*_0$.
\end{claim}
\begin{proof}[Proof of claim]
 Since $r\leq \pi_e(p)$ there is a sequence of objects $\langle \nu_0,\dots, \nu_{n-1}\rangle \in T^p$ such that $r\leq^* \pi_e(p)\cat\langle \nu_0\restriction e, \dots, \nu_{n-1}\restriction e\rangle$. Define
$$f^{r^*}:=f^r\cup \{\langle \alpha,\varnothing\rangle\mid \alpha\in\dom(f^p)\setminus \dom(f^r)\},$$
$$f^{p^*}:= (f^p\cat \vec\nu)\cup \{\langle\alpha,\varnothing\rangle\mid \alpha\in \dom(f^r)\setminus \dom(f^p)\},$$
$$T^{*}:=\pi^{-1}_{\dom(f^{p^*}),\dom(f^p)}T^p_{\vec\nu}\cap \pi^{-1}_{\dom(f^{p^*}),\dom(f^r)}T^r.$$
Clearly $p^*_1:=\langle f^{r^*}, T^*\rangle$ is a $\leq^*$-extension of $r$. Similarly, $p^*_0:=\langle f^{p^*}, T^*\rangle$ is a $\leq$-extension of $p$ as witnessed by $\vec\nu\in T^p$. Note that $f^{p^*}\restriction e=f^{r^*}\restriction e.$

\smallskip

Define $\gamma\colon s\leq^* p^*_0\cat \vec\eta\mapsto \langle f^s\setminus \dom(f^{r^*})\cup f^{r^*}\cat \vec\eta, T^s\rangle$. It is routine to check that this is an isomorphism witnessing $\mathbb{P}/p^*_0\simeq \mathbb{P}/p^*_1$. 

Let $G\s \mathbb{P}/p^*_0$ be a generic filter and consider $G\restriction e$ the $\mathbb{P}_e$-generic filter induced by $\{s\restriction e\mid s\in G\}$. Since $e\s \dom(f^p)\s \dom(f^{p^*})$ note that
$$\{s\restriction e\mid s\in G\}=\{\langle (f^{p^*}\cat\vec \eta)\restriction e, T^s\rangle\mid s\in G\}=\{\langle (f^{r^*}\cat\vec \eta)\restriction e, T^s\rangle\mid s\in G\}.$$
But the right-hand-side above equals $\{s\restriction e\mid s\in \gamma``G\}$.
Therefore, the $\mathbb{P}_e$-generics induced by both $G$ and $\gamma``G_0$ are the same. As a result, the $\mathbb{P}_e$-generics induced by $G$ and $\gamma``G_0$ via the projection $\pi_e\colon \mathbb{P}\rightarrow \mathbb{P}_e$ are the same. 
\end{proof}
 By the previous lemma there are conditions $p^*_0\leq p$ and $p^*_1\leq r$ and  an isomorphism $\gamma\colon \mathbb{P}/p^*_0\rightarrow\mathbb{P}/p^*_1$ with the above-stated property. 
By (1) we know that $ p^*_0\forces_{\mathbb{P}} \varphi(\dot b,\dot a_0,\dots, \dot a_n).$ We claim that $p^*_1$ forces the same. 
Towards a contradiction, assume $p^*_1\nVdash \varphi(\dot b,\dot a_0,\dots, \dot a_n)$ and let $x\leq p^*_1$ forcing $\neg \varphi(\dot b,\dot a_0,\dots, \dot a_n)$. Let $G$ a $\mathbb{P}$-generic filter with $\bar{x}:=\gamma^{-1}(x)\in G$. Note that $\bar x\leq p^*_0$, as $x\leq p^*_1$ and $\gamma$ is an isomorphism. Then, $V[G]\models \varphi(\dot{b}_{G}, (\dot{a}_0)_{G}, \dots (\dot{a}_{n})_{G})$ because $p^*_0$ forces this fact. Observe that $\dot{b}_{G}=\dot{b}_{\pi_e``G}$ simply because $\dot{b}$ is a $\mathbb{P}_e$-name. Clearly the same observation applies to the $(\dot{a}_i)_{G}$'s. Thus, 
$V[G]\models \varphi(\dot{b}_{\pi_e``G}, (\dot{a}_0)_{\pi_e``G}, \dots (\dot{a}_{n})_{\pi_e``G}).$
On the other hand, we assumed that $x$ forces $\neg \varphi(\dot b,\dot a_0,\dots, \dot a_n)$. So, 
$$V[\gamma``G]\models \neg \varphi(\dot{b}_{(\pi_e\circ \gamma)``G}, (\dot{a}_0)_{(\pi_e\circ \gamma)``G}, \dots (\dot{a}_{n})_{(\pi_e\circ \gamma)``G}).$$
However, $V[\gamma``G]=V[G]$ and $\pi_e``G=(\pi_e\circ \gamma)``G$ by the properties of the isomorphism $\gamma$, yielding we have the desired contradiction.
\end{proof}

\subsection{The $\Sigma$-Prikry tool box}

The core technology employed in this paper is the $\Sigma$-Prikry tool box developed by the authors in \cite{DPT}.
For later use in the manuscript we survey (without proofs) some of the concepts and   results proved in  \cite[\S3]{DPT}. Readers unfamiliar with this material are advised to have a copy of \cite{DPT} at their disposal. For the rest of the section we assume that 
$$\langle \pi_{e,d}\colon \mathbb{P}_e\rightarrow\mathbb{P}_d\mid d\s e\,\wedge\, e,d\in\mathcal{D}^\ast\rangle$$
is the commutative system of forcings and projections identified in Lemma \ref{lemma: quasi projections}.

\smallskip

We begin recalling the \emph{interpolation lemma} which is instrumental in the proofs of the forthcoming constellation lemma as well as in other arguments of this paper. 

\begin{lemma}[The Interpolation Lemma, {\cite[Lemma 3.4]{DPT}}]\label{lemma: interpolation}
    Let $G$ be a $\mathbb{P}$-generic, let $d\subsetneq e$ be domains in $\mathcal{D}^\ast$ and let $p\in\mathbb{P}/\pi_d`` G$. Then there is a $\mathbb{P}_e/\pi_d`` G$-generic filter $h\in V[G]$ such that $\pi_e(p)\in h$. \qed
\end{lemma}

\begin{lemma}[The Constellation Lemma, {\cite[Lemma 3.18]{DPT}}]\label{lemma: constellation lemma}
Let $G$ be a $\mathbb{P}$-generic filter, let $d\in\mathcal{D}^\ast$, let $h\in V[G]$ be $\mathbb{P}_d$-generic, and let $p\in\mathbb{P}/h$. 
   There is a $\mathbb{P}/h$-generic $G^\ast$ (over $V[h]$) with $p\in G^*$
    and $\mathcal{P}(\kappa)^{V[G^\ast]}=\mathcal{P}(\kappa)^{V[G]}.$\qed
\end{lemma}
A useful consequence of the above is the following:

\begin{cor}[{$L(\mathcal{P}(\kappa))$-capturing}]\label{cor: L(P(kappa))-capturing}
 Let $A\in V$, and let $E\in L(\mathcal{P}(\kappa)^{V[G]}, A)$. Suppose there is $d_0\in\mathcal{D}^\ast$ with $E\s V[\pi_{d_0}`` G]$. Then there is $d\in\mathcal{D}^\ast$ with $d_0\s d$ and $E\in V[\pi_d`` G]$.
\end{cor}
\begin{proof}
    First, since $E\in L(\mathcal{P}(\kappa)^{V[G]}, A)$ and $E \s V[\pi_{d_0}``G]$, there is a first-order formula $\psi(x,y_0,\dots, y_m)$ in the language of set theory and parameters $e_0,\dots, e_m$ in $\mathcal{P}(\kappa)^{V[G]} \cup A\cup \{A\}\cup \ord$ such that
    $$E=\{a\in V[\pi_{d_0}`` G]\mid V[G]\models \psi(a,e_0,\dots,e_m)\}.$$
    
    By the $\kappa$-capturing property, there is $d\in \mathcal{D}^*$ with $d_0\s d$ and $\{e_0\dots,e_m\}\in  V[\pi_{d}``G]$. Crucially, $\psi(x,y_0,\dots, y_{m})$ can be assumed to be absolute between generic extensions via $\mathbb{P}/\pi_{d}``G$ that share the same $\mathcal{P}(\kappa)$; namely, if $G_0$ and $G_1$ are $\mathbb{P}/\pi_{d}``G$-generic filters over $V[\pi_d``G]$ with $\mathcal{P}(\kappa)^{V[G_0]}=\mathcal{P}(\kappa)^{V[G_1]}$, then 
        $$V[G_0]\models \psi(a, e_0,\dots, e_{m})\;\Longleftrightarrow\; V[G_1]\models \psi(a, e_0,\dots, e_{m}).$$  
    \begin{claim}\label{claim: L(P(kappa))-capturing}
        $E=\{a\in V[\pi_{d_0}`` G]\mid  {V[\pi_d``G]}\models``\one\forces_{\mathbb{P}/\pi_{d}``G}\psi(\check{a},\check{e}_0,\dots,\check{e}_m)"\}$.
    \end{claim}
    \begin{proof}[Proof of claim]
        The right-to-left inclusion is evident. Conversely, let $a\in E$ and suppose, towards a contradiction, that there is $p\in \mathbb{P}/\pi_{d}``G$ such that \begin{equation}\label{eq star}
            V[\pi_d``G]\models ``p\forces_{\mathbb{P}/\pi_{d}``G}\neg\psi(\check{a},\check{e}_0,\dots,\check{e}_m)".
        \end{equation} Apply the \emph{Constellation Lemma} (Lemma~\ref{lemma: constellation lemma}) to the condition $p\in\mathbb{P}/\pi_{d}``G$. This lemma ensures the existence of a $\mathbb{P}/\pi_{d}``G$-generic $G^\ast$ such that $p\in G^\ast$ and $\mathcal{P}(\kappa)^{V[G^\ast]}=\mathcal{P}(\kappa)^{V[G]}.$
       By \eqref{eq star} above,  $V[G^\ast]\models\neg \psi(a, e_0,\dots, e_m)$. By absoluteness of $\psi(x,y_0,\dots, y_m)$, $V[G]\models\neg \psi(a,e_0,\dots,e_m)$. A contradiction with  $a\in E$.
    \end{proof}
    By the claim, $E$ is definable in $V[\pi_{d}``G]$ using the quotient forcing $\mathbb{P}/\pi_{d}``G$ and the parameters $e_0,\dots, e_m$. As a result, $E\in V[\pi_{d}``G]$. 
\end{proof}

\subsection{Combinatorial consequences of the $\kappa$-$\psp$.}\label{sec: combinatorial consequences of PSP}
 As anticipated in the Introduction ($\S\ref{sec: intro}$), we discuss a series of interesting combinatorial consequences of the $\kappa$-$\psp$. Accordingly, we briefly recall the relevant definitions. As usual, $\kappa$ will be a singular strong limit cardinal with $\cf(\kappa)=\omega$. Letting $\langle \kappa_n\mid n<\omega\rangle$ be an increasing sequence of regular cardinals, $\prod_{n<\omega} \kappa_n$ denotes the set of functions $f: \omega\rightarrow \sup_{n<\omega}\kappa_n$ such that $f(n)<\kappa_n$ for all $n<\omega$. If $f,g\in \prod_{n<\omega}\kappa_n$ we write $f<^\ast g$ whenever $\{n<\omega\mid f(n)\geq g(n)\}$ is finite. We define $f=^\ast g$ similarly.  A \emph{scale on} $\langle \kappa_n\mid n<\omega\rangle$ \emph{of length} $\mu$ is a $<^\ast$-increasing and $<^\ast$-cofinal sequence $\langle f_\beta\mid \beta<\mu\rangle$ of functions in $\prod_{n<\omega}\kappa_n$. A \emph{scale at} $\kappa$ is a scale $\langle f_\beta\mid \beta<\kappa^+\rangle$ on some $\langle \kappa_n\mid n<\omega\rangle$ converging to $\kappa$. On the other hand, $\diamondsuit_{\kappa^+}$ is the principle postulating  the existence of a sequence $\langle A_\alpha\mid \alpha < \kappa^+\rangle$ such that for every $A \s \kappa^+$ the set $\{\alpha <\kappa^+\mid A_\alpha = A\cap \alpha\}$ is stationary.  The $\sch_\kappa$ asserts that $\mathcal{P}(\kappa)$ is well-orderable and that $|\mathcal{P}(\kappa)|=\kappa^+$. Under $\zf$, the failure of $\diamondsuit_{\kappa^+}$ follows from $\neg\sch_\kappa$.
\begin{prop}[$\zf+\dc_\kappa$]\label{thm: combinatorics after psp} Assume that every subset of ${}^\omega \kappa$ has the $\kappa$-$\psp$. Then, there is no $\kappa^+$-sequence consisting of distinct members of $\mathcal{P}(\kappa)$. In particular, \begin{enumerate}
        \item\label{item 1: combinatorics after psp} there are no scales at $\kappa$, and
        \item\label{item 2: combinatorics after psp} both $\sch_\kappa$ and $\diamondsuit_{\kappa^+}$ fail.
        \item\label{item 3: combinatorics after psp} $\kappa^+$ is inaccessible in $L[a]$ for all $a\s\kappa$.\footnote{Note that by $\kappa$-capturing this is trivially true in the Merimovich's model defined in $\S\ref{subsec: Meromovich}$.}
        %\item There is no $\kappa^+$-sequence consisting of distinct members of $\mathcal{P}(\kappa)$.\seba{I think (3) implies (1) and (2)}
%        \item Both $\ads_\kappa$ and $\npt(\kappa^+,\aleph_1)$ fail.
    \end{enumerate}
    As a result, the above hold in $L(\mathcal{P}(\kappa))$ of the model in Theorem \ref{PSP in Solovay}.
\end{prop} 
%Proposition \ref{thm: combinatorics after psp} yields another analogy with the countable classical case. For instance, it is well known  that $\omega_1$ is inaccessible in $L[a]$ for all $a\s\omega$ provided that every set of reals has the $\psp$ (cf. \cite{Kan}). 
Before proving Proposition \ref{thm: combinatorics after psp}, some remarks are in order. 

$(\alpha)$ It is well known  that $\omega_1$ is inaccessible in $L[a]$ for all $a\s\omega$ provided that every set of reals has the $\psp$ (cf. \cite{Kan}). Thus, Proposition \ref{thm: combinatorics after psp} yields another analogy with the countable classical case.

$(\beta)$ A celebrated theorem of Shelah says that, in $\zfc$, every singular cardinal of countable cofinality admits a scale (see, e.g., \cite[p. 55]{Sh:g} or \cite[Theorem 2.26]{Abraham2010}). Shelah's proof, though, makes heavy use of the Axiom of Choice. In \cite[Theorem 9]{ShiTrang}, Shi--Trang proved that, under $I_0(\kappa)$, there are no scales at $\kappa$ in $L(\mathcal{P}(\kappa))$. Proposition \ref{thm: combinatorics after psp} combined with Theorem \ref{PSP in Solovay} shows that this configuration  can be consistently obtained from considerably weaker large cardinal assumptions.

$(\gamma)$ In \cite[Theorem 15]{ShiTrang}, Shi--Trang show that, under $I_0(\kappa)$, there is no $\kappa^+$-sequence of distinct elements of $\mathcal{P}(\kappa)$ in $L(\mathcal{P}(\kappa))$, yielding the failures of $\mathsf{SCH}_\kappa$ and $\diamondsuit_{\kappa^+}$ in $L(\mathcal{P}(\kappa))$. This is further evidence that $L(\mathcal{P}(\kappa))$ behaves under $I_0(\kappa)$ very much like $L(\mathbb{R})$ does under $\ad^{L(\mathbb{R})}$. Proposition \ref{thm: combinatorics after psp} shows that the same consequence can be obtained from the $\kappa$-$\psp$.

\begin{proof}[Proof of Proposition \ref{thm: combinatorics after psp}]
    Towards a contradiction, assume $\langle x_\alpha\mid \alpha<\kappa^+\rangle$ is a sequence of distinct subsets of $\kappa$. Appealing to the $\kappa$-$\psp$, there is an injection $\iota: {}^\kappa2\rightarrow \mathcal{P}(\kappa)$ with $\iota``({}^\kappa 2)\subseteq \{x_\alpha\mid \alpha<\kappa^+\}$. Define an injective map $h:{}^\kappa 2\to\kappa^+$ as $h(x)=\alpha$ if and only if $\iota(x)=x_\alpha$. This yields a well-ordering of ${}^\omega\kappa$ of length $\kappa^+$, and also a well-ordering of the set of all $\kappa$-perfect subsets of ${}^\omega\kappa$, say $\langle P_\alpha\mid \alpha< \kappa^+\rangle$.\footnote{Recall that $\kappa$ is a strong limit singular cardinal with $\cf(\kappa)=\omega$, so there is a bijection between ${}^\omega\kappa$ and ${}^\kappa 2$.} We build a subset of ${}^\omega\kappa$ without the $\kappa$-$\psp$ via a standard recursive construction (see \cite[Proposition~11.4]{Kan}). Specifically, define $\langle a_\alpha\mid \alpha<\kappa^+\rangle$ and $\langle b_\alpha\mid \alpha<\kappa^+\rangle$, with $a_\alpha\neq b_\alpha$ for all $\alpha<\kappa^+$, as follows. If $a_\beta$ and $b_\beta$ have been defined for $\beta<\alpha$, let $a_\alpha$ be the least $a$ in the well-ordering of ${}^\omega\kappa$ with $a\in P_\alpha \setminus(\{a_\beta\mid\beta<\alpha\}\cup\{b_\beta\mid\beta<\alpha\})$,
    and let $b_\alpha$ be the least $b$ in the well-ordering of ${}^\omega\kappa$ with
    $b\in P_\alpha \setminus(\{a_\beta\mid\beta\leq \alpha\}\cup\{b_\beta\mid\beta<\alpha\})$. Then $\{a_\alpha\mid\alpha<\kappa^+\}$ does not have the $\kappa$-$\psp$, yielding the claimed contradiction. 

    \eqref{item 1: combinatorics after psp} The fact that there are no scales at $\kappa$ follows easily, since every scale $\langle f_\beta\mid \beta<\kappa^+\rangle$ can canonically be coded as a $\kappa^+$-sequence of distinct members of $\mathcal{P}(\kappa)$.

    \eqref{item 2: combinatorics after psp} The failure of $\sch_\kappa$ is obvious.

    \eqref{item 3: combinatorics after psp} Aiming for a contradiction, pick $a\s\kappa$ and $\alpha\leq\kappa$ such that $L[a]\models``2^\alpha\geq\lambda"$, where $\lambda=\kappa^+$. In particular, $L[a]\models``2^\kappa\geq\lambda"$ but this absurd since there cannot be $\kappa^+$ many distinct members of $\mathcal{P}(\kappa)$.
\end{proof}

The above yields the following consistency result, which partially answers a question of Dimonte in \cite{VD}.
\begin{cor}
    The consistency of $$\zfc+\exists\kappa\, \exists\lambda\, (\lambda \text{ is inaccessible and }\kappa \text{ is ${<}\lambda$-supercompact})$$ implies the consistency of $\zf+\text{ There is a strong cardinal $\kappa$ of countable cofinality}$ such that \begin{enumerate}
        \item $\dc_\kappa$ holds.
        \item There is no $\kappa^+$-sequence consisting of distinct members of $\mathcal{P}(\kappa)$.
    \end{enumerate}
\end{cor}
\begin{question}
    What is the consistency strength of $$\zf+\exists\kappa\, (\kappa \text{ is singular and }L(\mathcal{P}(\kappa))\models\neg\ac)\text{?}$$
\end{question}

\section{Combinatorics 
in $L(\mathcal{P}(\kappa))$ of Merimovich's model}\label{sec: Combinatorics in Merimovich}
In this section we would like to prove a few combinatorial properties that are specific of $L(\mathcal{P}(\kappa))^{V[G]}$ where $V[G]$ is Merimovich's model. These will enrich the list of combinatorial properties already isolated in Theorem \ref{thm: combinatorics after psp}.
\begin{theorem}
    Assume that $\kappa$ is supercompact and that $\lambda>\kappa$ is inaccessible. Let  $\mathbb{P}$ be Merimovich's forcing from \S\ref{sec: Merimovich forcing} and $G\s \mathbb{P}$ be generic. Then, the following properties hold in $L(\mathcal{P}(\kappa))^{V[G]}$:
    \begin{enumerate}
        \item $\AP_\kappa$ fails.
        \item $\kappa\xrightarrow[]{\mathrm{OD}} (\omega)^\omega_{V_\mu}$ holds for all $\mu<\kappa$.
           \item Assuming that $\lambda$ is weakly compact in $V$, every $\kappa^+$-tree
           in $L(\mathcal{P}(\kappa))^{V[G]}$ contains a cofinal branch, possibly not in $L(\mathcal{P}(\kappa))^{V[G]}$.
    \end{enumerate}
    Moreover, if $\lambda$ is measurable as witnessed by $U$, then  in the $\dc_\kappa$-model, $L(\mathcal{P}(\kappa)^{V[G]},U)$, properties $(1)$--$(3)$  hold together with $``(\kappa^+)^{V[G]}$ is measurable".

    In addition, in a $\lambda$-distributive generic extension of $V[G]$, 
    $$\mathrm{Cub}_\lambda\restriction (E^\lambda_{\mathrm{Reg}})^V$$
    is a $L(\mathcal{P}(\kappa))^{V[G]}$-ultrafilter.
\end{theorem}

\subsection{Shelah's Approachability Property}

We are interested in showing that the approachability property fails at $\kappa$ in $L(\mathcal{P}(\kappa))^{V[G]}$. This is a very weak property --thus, its negation is very strong-- and was introduced by Shelah in \cite{Sh:108}. 

\begin{comment}
\begin{definition}
    Let $\mu$ be a regular cardinal and $a=\langle a_\alpha: \alpha<\mu\rangle$ a sequence of bounded subsets of $\mu$. A limit ordinal $\alpha<\mu$ is said to be \emph{approachable with respect to} $a$ if there is an unbounded set $A\subseteq\alpha$ with order type $\mathrm{cof}(\alpha)$ such that 
    \[\{A\cap \beta:\beta<\alpha\}\subseteq\{a_\beta:\beta<\alpha\}.\]
\end{definition}

\begin{definition}
    Let $\mu$ be a regular cardinal. Define an ideal $I[\mu]$ on $\mu$ as follows: $S\subseteq\mu$ is in $I[\mu]$ if and only if there is a sequence $a$ of bounded subsets of $\mu$ and a club $C\subseteq\mu$ such that every $\alpha\in C\cap S$ is approachable with respect to $a$.
\end{definition}

This ideal $I[\mu]$ is called the \emph{approachability ideal on} $\mu$.

\begin{definition}
    Let $\kappa$ be any cardinal. We say that the \emph{approachability property} holds at $\kappa$ ($\AP_\kappa$, for short) if $\kappa^+\in I[\kappa^+]$ (hence, $I[\kappa^+]$ is an improper ideal).
\end{definition}

\end{comment}

\begin{definition}
    A sequence $\langle C_\alpha \mid \alpha < \kappa^+ \rangle$ is called an $\AP_\kappa$-\emph{sequence} if for each limit ordinal $\alpha<\kappa^+$, $C_\alpha\s\alpha$ is a club with $\otp(C_\alpha) = \cf(\alpha)$, and there is a club $D\s \kappa^+$ such that $\forall \alpha \in D \ \forall \beta < \alpha \ \exists \gamma < \alpha \ (C_\alpha \cap \beta = C_\gamma)$.  We say that the \emph{approachability property} holds at $\kappa$ ($\AP_\kappa$, for short) if there is an $\AP_\kappa$-sequence
\end{definition}

The $\AP_\kappa$ has been extensively studied by Shelah \cite{Sh:108} and by many other authors, including Cummings–Foreman–Magidor \cite{CumForMag}, Gitik–Krueger \cite{GitikKrueger}, Rinot \cite{RinotAP}, Cummings et al. \cite{CummingsAP}, and more recently by Jakob-Levine \cite{JakobLevine} and Jakob–Poveda \cite{JakobPoveda}. There are various equivalent formulations of $\ap_\kappa$ (see \cite{MR2768694}) but the above is the one that will fit better our arguments.

\smallskip

We begin with a technical application of the Interpolation Lemma \ref{lemma: interpolation}:

\begin{lemma}\label{the weak Perfect Set Lemma}
     Let $G$ be $\mathbb{P}$-generic. Suppose that $d,e\in\mathcal{D}^\ast$, $d\subsetneq e$, $q\in \pi_e``G$, $(\kappa^+)^{V[\pi_d``G]}<\alpha< (\kappa^+)^{V[G]}$, and $C\in\mathcal{P}(\alpha)^{V[\pi_e``G]}\setminus V[\pi_d``G]$. If $\dot{C}$ is a $\mathbb{P}_e/\pi_d``G$-name for $C$ and $q\Vdash_{\mathbb{P}_e/\pi_d``G} \dot{C}\subseteq\check{\alpha}$, then there is a $\mathbb{P}_e/\pi_d``G$-generic filter $h\in V[G]$ (over $V[\pi_d``G]$) with $q\in h$ and $\dot{C}_h\neq C$.
 \end{lemma}
 \begin{proof}
 For the sake of readability we denote $\Vdash_{\mathbb{P}_e/\pi_d``G}$ just by $\Vdash$. No confusion should arise as $\mathbb{P}_e/\pi_d``G$ is the only forcing  involved in the argument.

 Since $\one_{\mathbb{P}_e}\Vdash``\dot{C}\notin V[\pi_d``G]$", there is $\alpha_0<\alpha$ such that $q$ does not decide the statement $``\check{\alpha}_0\in\dot{C}$". Thus there are conditions $q_0$ and $q_1$ in $\mathbb{P}_e/\pi_d``G$, below $q$, such that $q_0\Vdash\check{\alpha}_0\in\dot{C}$ and $q_1\Vdash\check{\alpha}_0\notin\dot{C}$. Two cases may occur:

 \smallskip

 \textbf{Case 1:} Suppose $\alpha_0\in C$. Apply the \emph{Interpolation Lemma} (Lemma~\ref{lemma: interpolation}) to
 \begin{displaymath}
        \begin{tikzcd}
   \mathbb{P} \arrow{r}{\pi_e} \arrow[bend right]{rr}{\pi_d} & \mathbb{P}_e \arrow{r}{\pi_{e,d}}  & \mathbb{P}_d
 \end{tikzcd}
     \end{displaymath}
 $q_1\in\mathbb{P}_e/\pi_d``G\s  \mathbb{P}/\pi_d``G$ to find a $\mathbb{P}_e/\pi_d``G$-generic filter $h\in V[G]$ such that $\pi_e(q_1)=q_1\in h$. Since $q_1\Vdash\check{\alpha}_0\notin\dot{C}$ it follows that $\dot{C}_h\neq C$.

 \smallskip

 \textbf{Case 2:} Suppose that $\alpha_0\notin C$. Apply the \emph{Interpolation Lemma} to $q_0$ to produce a $\mathbb{P}_e/\pi_d``G$-generic filter $h\in V[G]$ with $q_0\in h$. As $q_0\Vdash\check{\alpha}_0\in\dot{C}$ it follows that $\dot{C}_h\neq C$.
 \end{proof}

\begin{theorem}\label{thm: not ap}
$\one \Vdash_{\mathbb{P}}``\text{$L(\mathcal{P}(\kappa))\models \neg \AP_{\kappa}$"}.$ 
\end{theorem}
\begin{proof}
    Let $G$ be $\mathbb{P}$-generic. Suppose, towards a contradiction, that $\vec{C}=\langle C_\alpha\mid\alpha<(\kappa^+)^{V[G]}\rangle$ is an $\AP_{\kappa}$-sequence in  $L(\mathcal{P}(\kappa))^{V[G]}$. Then $\vec{C}$ can be written as $$\vec{C}=\{\langle\alpha, C\rangle\mid V[G]\models \varphi(\alpha, C,a_0,\dots,a_n)\},$$ where $\varphi(x,z,y_0,\dots, y_n)$ is a first-order formula in the language of set theory and $a_0,\dots, a_n$ are parameters in $\mathcal{P}(\kappa)^{V[G]}\cup\ord$.
Moreover, $\varphi(x,y_0,\dots, y_{n})$ is absolute between generic extensions via $\mathbb{P}$ sharing the same $\mathcal{P}(\kappa)$ (or, equivalently, the same $V_{\kappa+1}$).  By $\kappa$-capturing, there is $d\in\mathcal{D}^\ast$ with $\{a_0,\dots, a_n\}\s V[\pi_d``G]$.

    \begin{claim}
       There is $\alpha<(\kappa^+)^{V[G]}$ with $C_\alpha\notin V[\pi_d``G]$.
   \end{claim}
   \begin{proof}[Proof of claim]
       Since $\vec{C}$ is an $\AP_\kappa$-sequence and $\kappa$ is singular in $V[G]$, it follows that $\otp(C_\alpha)=\cf(\alpha)^{V[G]}<\kappa$, for each limit ordinal $\alpha<(\kappa^+)^{V[G]}$. Let $\alpha$ be a regular cardinal in $V[\pi_d``G]$ with $(\kappa^+)^{V[\pi_d``G]}<\alpha<(\kappa^+)^{V[G]}$. We claim that $C_\alpha\notin V[\pi_d``G]$. If not, the $V[\pi_d``G]$-regularity of $\alpha$ yields $$\kappa<(\kappa^+)^{V[\pi_d``G]}<\alpha=|C_\alpha|^{V[\pi_d``G]}\leq\otp(C_\alpha)=\cf(\alpha)^{V[G]},$$ contradicting the fact that $\cf(\alpha)^{V[G]}<\kappa$.
   \end{proof} 
   
   Let $C_\alpha\notin V[\pi_d``G]$. By ${<}\lambda$-capturing, there is an $e\in \mathcal{D}^*$ with $d\subsetneq e$ and $C_\alpha\in V[\pi_e``G]$. Since $\langle\alpha, C_\alpha\rangle\in \vec{C}$, we have $V[G]\models \varphi(\alpha, C_\alpha,a_0,\dots,a_n)$.
   All the above parameters belong to $V[\pi_e``G]$ so there is $p\in G$  such that \begin{equation}\label{eq AP I}
       V[\pi_e``G]\models ``p\Vdash_{\mathbb{P}/\pi_e``G} \text{$\varphi(\check{\alpha}, \check{C}_\alpha, \check{a}_0,\dots, \check{a}_n)$''.}
   \end{equation} 
Denote by $\phi(p,\mathbb{P}, \pi_e``G, \alpha, C_\alpha,a_0,\dots,a_n)$ the formula $$``\text{$p\forces_{\mathbb{P}/\pi_e``G}\varphi(\check{\alpha}, \check{C}_\alpha, \check{a}_0,\dots, \check{a}_n)$''.}$$ Since $V[\pi_e``G]=V[\pi_d``G][\pi_e``G]$, there is $q\in\pi_e``G$ with $q\leq\pi_e(p)$ such that 
   \begin{equation}\label{eq 3 AP: q forces the formula Phi}
       V[\pi_d``G]\models ``\text{$q\forces_{\mathbb{P}_e/\pi_d``G}\phi(\check{p},\check{\mathbb{P}}, \dot{G}, \check\alpha, \check C_\alpha,\check a_0,\dots,\check a_n)$''},
   \end{equation} 
   where $\dot{C}$ is a $\mathbb{P}_e/\pi_d``G$-name with $\dot{C}_{\pi_e``G}=C_\alpha$, and $\dot{G}$ is the canonical name for the generic. Without loss of generality we may assume that $q=\pi_e(p')$ for some  $p'\in G$ below $p$. Moreover, since $C_\alpha\in V[\pi_e``G]$, by extending $q$ (inside $\pi_e``G$) if necessary we may assume that $q\Vdash_{\mathbb{P}_e/\pi_d``G} \dot{C}\subseteq\check{\alpha}.$ Equation~\eqref{eq 3 AP: q forces the formula Phi} implies that 
   $$V[\pi_d``G][h]\models ``\text{$\phi(p,\mathbb{P}, h,\alpha, \dot{C}_h, a_0,\dots,a_n)$''}$$
   or, equivalently, 
   \begin{equation}\label{eq 4 AP: forcing over V[h]}
     V[\pi_d``G][h]\models ``\text{$p\forces_{\mathbb{P}/h}\varphi( \check{\alpha}, \check{\dot{C}}_h, \check{a}_0,\dots, \check{a}_n)$''}.
   \end{equation}
    provided  $h$ is a $\mathbb{P}_e/\pi_d``G$-generic filter (over $V[\pi_d``G]$) containing $q$. By Lemma \ref{the weak Perfect Set Lemma}, there is a $\mathbb{P}_e/\pi_d``G$-generic filter $h\in V[G]$ (over $V[\pi_d``G]$) containing $q$, such that $\dot{C}_h\neq\dot{C}_{\pi_e``G}=C_\alpha$. Next we argue that $\langle \alpha, \dot{C}_h\rangle\in \vec{C}$ which will yield a contradiction. Indeed, if that were to be the case, since $\vec{C}$ is a function, then  $\dot{C}_h=C_\alpha$, but this is false. To see $\langle \alpha, \dot{C}_h\rangle\in \vec{C}$, apply the \emph{Constellation Lemma} (i.e., Lemma~\ref{lemma: constellation lemma}) with respect to the condition $p\in\mathbb{P}/h$. This lemma gives us a $\mathbb{P}/h$-generic filter $G^\ast$ (over $V[h]$) such that $p\in G^\ast$ and  $\mathcal{P}(\kappa)^{V[G^\ast]}=\mathcal{P}(\kappa)^{V[G]}$. Equation~\eqref{eq 4 AP: forcing over V[h]} yields $V[G^\ast]\models\varphi( \alpha, \dot{C}_h, a_0,\dots, a_n)$ and so by absoluteness of $\varphi$ we obtain
   $V[G]\models \varphi(\alpha, \dot{C}_h, a_0,\dots, a_n)$. Therefore $\langle \alpha, \dot{C}_h\rangle\in \vec{C}$.
\end{proof}

\subsection{Existence of $\kappa$-MAD families}
In his seminal paper \cite{HappyFamilies}, Mathias showed that there are no infinite analytic maximal almost disjoint families in $\mathcal{P}(\omega)$. Mathias also asked if infinite mad families can exist in the Solovay model. This was recently answered in the negative by  Törnquist in   \cite{Tornquist}. Here, we would like to draw a difference between the original classical model and ours by showing that the latter encompasses $\kappa$-mad families in $\mathcal{P}(\kappa)$ of size $\kappa$. Recall that this is a family $\{A_\alpha\mid \alpha <\kappa\}$ consisting of (distinct) sets $A_\alpha \s [\kappa]^\kappa$ such that $|A_\alpha \cap A_\beta|<\kappa$ and for every set $B\in [\kappa]^\kappa$ there is some $\alpha< \kappa$ such that $|A_\alpha \cap B|=\kappa$.

\begin{theorem}
    There is a $\kappa$-mad family of size $\kappa$ in $L(\mathcal{P}(\kappa))^{V[G]}$.
\end{theorem}
\begin{proof}
Since $\kappa$ is singular strong limit the Erd\H{o}s--Hechler theorem \cite[Theorem 2.4]{ErdosHechler} gives a $\kappa$-mad family $\{A_\alpha\mid \alpha <\kappa\}\s [\kappa]^\kappa$ of size $\kappa$. Let $\psi\colon \kappa \to \kappa\times \kappa$ be any bijection in $V[G]$ and consider the set $c = \{\beta<\kappa\mid \exists \alpha < \kappa\,\exists \xi\in A_\alpha\, \psi(\beta)=(\alpha,\xi)\}$ that codes the $\kappa$-mad family as a subset of $\kappa$. Then, $\psi, c\in L(\mathcal{P}(\kappa))^{V[G]}$ and so the $\kappa$-mad family is in $L(\mathcal{P}(\kappa))^{V[G]}$. Note that $\{A_\alpha \mid \alpha <\kappa\}$ remains $\kappa$-mad in $L(\mathcal{P}(\kappa))^{V[G]}$ by absoluteness.
\end{proof}

\subsection{A definable partition relation}
The partition relation $\kappa\xrightarrow[]{\mathrm{OD}} (\omega)^\omega_{V_\mu}$ establishes that any $\OD$-coloring  $c\colon [\kappa]^{\omega}\to V_\mu$ admits a set $H\in [\kappa]^\omega$ and a color $x\in V_\mu$ such that $c(s)=x$ for all $s\in [H]^\omega$.  In \cite{Kafkoulis1994-KAFTCS}, Kafkoulis proved the consistency of this partition relation with $\mathrm{ZF}+\mathrm{DC}_\kappa$ starting with a supercompact cardinal  and an inaccessible above. Kafkoulis' model is of the form $V(\bigcup_{\alpha<\kappa}\mathcal{P}(\kappa)^{V[G\restriction\alpha]})$ where $V[G]$ is a supercompact Prikry extension where an inaccessible $\lambda$ is collapsed to $\kappa$.

One limitation of Kafkoulis' model is that it does not contain $\mathcal{P}(\kappa)^{V[G]}$. In this section we prove, building upon Kafkoulis' argument,  that the above-mentioned partition relation holds in $L(\mathcal{P}(\kappa))$ of the Merimovich's extension. In our case there is an additional complication that the proof will feature -- the $\leq^*$ order of $\mathbb{P}$ may inadvertently produce finite bits for new Prikry sequences. This is not the case in Kafkoulis' situation where Supercompact Prikry is used.

\begin{theorem}\label{thm: kafkoulis}
    In $L(\mathcal{P}(\kappa))^{V[G]}$ the partition  $\kappa\xrightarrow[]{\mathrm{OD}} (\omega)^\omega_{V_\mu}$ holds for all $\mu<\kappa$.
\end{theorem}
\begin{proof}
    Let  $c\colon [\kappa]^{\omega}\to V_\mu$ be an $\mathrm{Ord}$-definable coloring in $L(\mathcal{P}(\kappa))^{V[G]}$. Then, 
    $$c=\{(s,x)\mid ([\kappa]^{\omega}\times V_\mu)^{V[G]}\mid V[G]\models \varphi(s,x,\alpha)\}$$
    for some %$a\in \mathcal{P}(\kappa)^{V[G]}$, 
    $\alpha\in \ord$ and $\varphi(\cdot, \cdot, \cdot)$ a formula that is absolute between models of ZFC containing the same subsets of $\kappa$. Note that $V_\mu^{V[G]}$ is just $V_\mu$.

    \smallskip

    Let $p\in G$ be a condition that forces (in $\mathbb{P}$) the formula $$\text{$``\forall s\in [\kappa]^{\omega}\exists! x\in V_\mu\,\varphi(s, x,\check{\alpha})$".}$$
    Denote $\dot{s}_\kappa$ the standard $\mathbb{P}$-name for the generic introduced by $\mathbb{P}_{\{\kappa\}}$.
    \begin{claim}
        The following set of conditions  is dense below $\pi_{\{\kappa\}}(p)$:$$\{\langle t,T\rangle\in \mathbb{P}_{\{\kappa\}}\mid \exists p^*\leq p\,\exists x^*\in V_\mu\; \left(\pi_{\{\kappa\}}(p^*)=\langle t, T\rangle\,\wedge\, p^*\forces_{\mathbb{P}} \varphi(\dot{s}_\kappa\setminus \max(t)+1, \check{x}^*,\check{\alpha})\right)\}.$$
    \end{claim}
    \begin{proof}[Proof of claim]
      Let $\langle t, T\rangle\leq_{\mathbb{P}_{\{\kappa\}}} \pi_{\{\kappa\}}(p)$ be arbitrary. Since $\pi_{\{\kappa\}}$ is a projection, there is a condition $q\leq p$ such that $\pi_{\{\kappa\}}(q)\leq_{\mathbb{P}_{\{\kappa\}}} \langle t, T\rangle$. Say $\pi_{\{\kappa\}}(q)=\langle s, S\rangle$.
      
      Consider the collection of formulas $\langle \Phi_x\mid x\in V_\mu\rangle$ where
       $$\Phi_x\equiv \varphi(\dot{s}_\kappa \setminus \max(s)+1, \check{x},\check{\alpha}).$$
       Since $\langle \mathbb{P},\leq^*\rangle$ is $\kappa$-closed, applying the Prikry property $|V_\mu|^+$-many times we find a condition $p^*\leq^*q$ such that $p^*\parallel_{\mathbb{P}} \Phi_x$ for all $x\in V_\mu$. Since $p^*\leq p$ in fact there is some $x^*\in V_\mu$ such that $p^*\forces_{\mathbb{P}}\Phi_{x^*}$. Note that the stems of $\pi_{\{\kappa\}}(p^*)$ and $\pi_{\{\kappa\}}(q)$ are the same precisely because $p^*\leq^* q$.\footnote{And, crucially, because $\kappa$ is always a member of the domain of the Cohen part of any condition!} As a result, $p^*$ and $x^*$ are as wanted.
    \end{proof}

    Let $p^*\in \mathbb{P}/G_{\{\kappa\}}$, $p^*\leq p$, and $x^*\in V_\mu$ be  such that $$p^*\forces_{\mathbb{P}} \varphi(\dot{s}_\kappa\setminus \max(f^{p^*}(\kappa))+1, x,\check{\alpha}).$$ 
    \begin{claim}
        $c(s_\kappa\setminus \max(f^{p^*}(\kappa))+1)=x^*$.
    \end{claim}
    \begin{proof}
       The constellation lemma gives a $\mathbb{P}/G_{\{\kappa\}}$-generic filter $G^*$ with $p^*\in G^*$ and $\mathcal{P}(\kappa)^{V[G^*]}=\mathcal{P}(\kappa)^{V[G]}$. Hence, $V[G^*]\models \varphi(s_\kappa\setminus \max(f^{p^*}(\kappa))+1,x,\alpha)$ because $(\dot{s}_\kappa)_G=(\dot{s}_\kappa)_{G^*}$.  By absoluteness of $\varphi(\cdot, \cdot, \cdot)$ we conclude that the same formula holds in $V[G]$, therefore $c(s_\kappa\setminus \max(f^{p^*}(\kappa))+1)=x^*$.
    \end{proof}

    We claim that for all $s\in [s_\kappa\setminus \max(f^{p^*}(\kappa))+1]^\omega$ the color of $s$ is also $x^*$.
    \begin{claim}
        $c(s)=x^*$ for all $s\in [s_\kappa\setminus \max(f^{p^*}(\kappa))+1]^\omega$ in $V[G]$.
    \end{claim}
    \begin{proof}
        Let $F_s=\{\langle t, T\rangle\in \mathbb{P}_{\{\kappa\}}\mid t\sqsubseteq f^{p^*}(\kappa)^\smallfrown s\,\wedge\, (f^{p^*}(\kappa)^\smallfrown s\setminus \max(t)+1)\s T\}$.

        By the Mathias criterion of genericity, $F_s$ is generic for $\mathbb{P}_{\{\kappa\}}$. Note that $p^*$ belongs to $ \mathbb{P}/F_s$: First, $p^*\in \mathbb{P}/G_{\{\kappa\}}$. Thus, the tree part of $\pi_{\{\kappa\}}(p^*)$ (call it $T$) contains $s_\kappa\setminus \max(f^{p^*}(\kappa))+1$. In particular,  $s\s T$. Thus, $p^*\in \mathbb{P}/F_s$.

        Applying  the constellation lemma we find a $\mathbb{P}/F_s$-generic  $G^*$ such that $p^*\in G^*$ and computing the same power set of $\kappa$ as $V[G]$. Arguing as before we conclude that $V[G]\models \varphi((\dot{s}_\kappa)_{G^*}\setminus \max(f^{p^*}(\kappa))+1,x,\alpha)$. But note that $(\dot{s}_\kappa)_{G^*}$ is precisely $s$ because $G^*$ is $\mathbb{P}/F_s$-generic and $f^{p^*}(\kappa)^\smallfrown s$ is the Prikry sequence induced by $F_s$.
    \end{proof}
    The above completes the proof of the theorem.
\end{proof}

\subsection{The complete $\vec{\mathcal{U}}$-Ramsey property}\label{sec: the complete Ramsey}

%\begin{definition}
 %   Let $O\s \prod_{n<\omega}\kappa_n$ be an open set in the $\vec{\mathcal{U}}$-$\textsf{EP}$ topology. Given a condition $p\in \mathbb{P}(\vec{\mathcal{U}})$ and a stem $t\in \prod \vec{A}^p$ we say that \emph{$p$ accepts $t$} if there is $q\leq^* p\cat t$ such that $\mathcal{N}_q\s O$. We say that \emph{$p$ rejects $t$} if $p$ does not accept $t$.
%\end{definition}
In this section we assume that we are given an increasing sequence $\langle \kappa_n\mid n<\omega\rangle$  of ${<}\lambda$-supercompact cardinals where $\lambda$ is an inaccessible cardinal above $\sup_{n<\omega}\kappa_n$. Let $\mathbb{P}^{\mathrm{diag}}$ be the \emph{Diagonal Supercompact Extender Based Prikry} from \cite[\S4.2]{DPT}. We will show that $\mathbb{P}^{\mathrm{diag}}$ forces 
\begin{center}
  ``Every $A\in \mathcal{P}(\prod_{n<\omega}\kappa_n)\cap L(\mathcal{P}(\kappa))^{V^{\mathbb{P}^{\mathrm{diag}}}}$ has the complete $\vec{\mathcal{U}}$-Ramsey property",  
\end{center}
where $\vec{\mathcal{U}}:=\langle U_n\mid n<\omega\rangle$ is the sequence of normal measures such that each $U_n\s\mathcal{P}(\kappa_n)$ is induced by some (fixed) elementary embedding $j_n\colon V\to M_n$ witnessing the ${<}\lambda$-supercompactness of $\kappa_n$.

\smallskip

What we intend to prove is the singular analogue of Ellentuck's theorem \cite{Ellentuck} asserting that every set of reals with the Baire property is completely Ramsey.  To prove it, we will bear on  \cite[Corollary 4.12]{DPT} which establishes the $\vec{\mathcal{U}}$-Baire property of every subset of $\prod_{n<\omega}\kappa_n$ in $L(\mathcal{P}(\kappa))$ of ${V^{\mathbb{P}^{\mathrm{diag}}}}$.

\smallskip

%As shown in \cite[]{DPT}, $\mathbb{P}^{\mathrm{diag}}$ projects onto the corresponding Diagonal Prikry forcing $\mathbb{P}(\vec{\mathcal{U}}).$

Through this section we denote by $\mathbb{P}(\vec{\mathcal{U}})$ the Diagonal Prikry forcing with respect to $\vec{\mathcal{U}}$ as presented in Definition~\ref{def: diagonal prikry}. Our primary topological space will be $\prod_{n<\omega}\kappa_n$ endowed with the $\UEP$ topology (see \S\ref{sec: gdst}). The final theorem will be obtained as a combination of a series of technical lemmas, the first of which is the following:

\begin{lemma}\label{lemma: diagonalization in Ramsey}
    Let $O\s \prod_{n<\omega}\kappa_n$ be $\UEP$-open and $p\in \mathbb{P}(\vec{\mathcal{U}})$. There is $p^*\leq^* p$ such that if there is $q\leq p^*$ with $\mathcal{N}_q\s O$ then $\mathcal{N}_{w(p,q)}\s O$.\footnote{Here $w(p,q)$ is the notation used in \cite{PartI} to refer to the $\leq$-least condition $r\in \mathbb{P}(\vec{\mathcal{U}})$ such that $q\leq^* r\leq p$.}
\end{lemma}
\begin{proof}
  We define a $\leq^*$-decreasing sequence of conditions $\langle p_n\mid n<\omega\rangle$ as follows. 
  
  Set $p_0:=p$. Suppose that $p_n$ has been defined. For each stem $t\in \prod_{i=\ell(p)}^{\ell(p)+n} {A}^{p^n}_i$ is there is  $q\leq^*p_n\cat t$ such that $\mathcal{N}_q\s O$ then we let $\vec{A}^{n,t}$ be the measure one sets of some of such $q$. Otherwise, we  let $\vec{A}^{n,t}:=\vec{A}^{p_n}\restriction [\ell(p)+n,\omega)$. We complete the induction step taking 
  $$A^{p_{n+1}}_i:=A_i^{p_n}\cap \{A^{n,t}_i\mid t\in \prod_{i=\ell(p)}^{\ell(p)+n}{A}^{p_n}_i\,\text{and $|t|=n$}\}\;\;\text{for all $i\geq \ell(p)+n+1$}$$
and setting $p_{n+1}$ be the condition with the same stem as $p$, the measure one sets of $A^{p_n}$ up to coordinate $\ell(p)+n$ and $A^{p_{n+1}}_i$ onwards. After this process we have constructed a $\leq^*$-decreasing sequence $\langle p_n\mid n<\omega\rangle$ and we take $p^*$ to be a $\leq^*$-lower bound for it. It is routine to check that $p^*$ has the intended properties.
\end{proof}

The next lemma is the $\mathbb{P}(\vec{\mathcal{U}})$-analogue of the classical Galvin--Prikry lemma \cite{GalvinPrikry} asserting that dense open sets are completely Ramsey:

\begin{lemma}\label{lemma: dense open sets are Ramsey}\label{lemma: Open sets are completely ramsey}
    Let $O\s \prod_{n<\omega}\kappa_n$ be an open set in the $\vec{\mathcal{U}}$-$\textsf{EP}$ topology. Then, for each $p\in \mathbb{P}(\vec{\mathcal{U}})$ there is a condition $q\leq^* p$ such that either $\mathcal{N}_q\s O$ or $\mathcal{N}_q\cap O=\emptyset$.

    In addition, if $O$ is dense in $\UEP$, there is $q\leq^* p$ such that $\mathcal{N}_q\s O$.
\end{lemma}
\begin{proof}
 Let $p\in \mathbb{P}(\vec{\mathcal{U}})$ and apply Lemma~\ref{lemma: diagonalization in Ramsey} to find a condition $p_0\leq^* p$ such that  if  $q\leq p_0$ is such that $\mathcal{N}_q\s O$ then $\mathcal{N}_{w(p_0,q)}\s O$. If there is $q\leq^* p_0$ such that $\mathcal{N}_q\s O$ then we take it. So, suppose there is no such a pure extension $q$.

 \smallskip

 Just to simplify notations let us agree that the length of $p_0$ is zero.

 \begin{claim}
    The set $\{\alpha\in A^{p_0}_0\mid \exists q\leq^* p_0\cat \langle \alpha\rangle\,\wedge\, \mathcal{N}_q\s O\}$ is not $\mathcal{U}_0$-large.
 \end{claim}
 \begin{proof}[Proof of claim]
   Suppose it is $\mathcal{U}_0$-large and call it $X$. Then, for each $\alpha\in X$ we find $q(\alpha)\leq^* p_0\cat \langle \alpha\rangle$ such that $\mathcal{N}_{q(\alpha)}\s O$. Let us define an auxiliary condition $a\leq^* p_0$ that has as a first measure one set $X$, and for coordinates $i\geq 1$, $$\textstyle A^{a}_i = \bigcap_{\alpha <\kappa_0} A^{q(\alpha)}_i.$$

  Let $x\in \mathcal{N}_a$. Then, $x(i)\in A^a_i$ for all $i\geq 1$. In particular, $x(i)\in A^{q(x(0))}_i$ for all $i\geq 1$. This is equivalent to saying that  $x\in  \mathcal{N}_{q(x(0))}\s O$. Ergo, $\mathcal{N}_a\s O$. This is a contradiction as we were assuming that no $a\leq^* p_0$ had such property.
 \end{proof}
Let $A^{p_1}_0$ be the complement of the set displayed in the above lemma. 
 Let $p_1$ be the $\leq^*$-extension of $p_0$ obtained after shrinking the first measure one set of $p_0$ to $A^{p_1}_0$ and leave the other measure one sets untouched.

   \begin{claim}
   The following set is  $\mathcal{U}_0$-large: 
      $$\{\alpha\in A^{p_1}_0\mid \{\beta\in A^{p_0}_1\mid \alpha<\beta\,\wedge\, \exists q\leq^* p_1\cat \langle \alpha,\beta\rangle \,\wedge\,\mathcal{N}_q\s O\}\notin \mathcal{U}_1\}.$$
   \end{claim}
   \begin{proof}[Proof of claim]
      Suppose that this was not the case. Then, for $\mathcal{U}_0$-many $\alpha$'s the set $\{\beta\in A^{p_0}_1\mid \alpha<\beta\,\wedge\, \exists q\leq^* p_1\cat \langle \alpha,\beta\rangle \,\wedge\,\mathcal{N}_q\s O\}$ is  $\mathcal{U}_1$-large. For each $\alpha\in A^{p_1}_0$ and $\beta\in A^{p_0}_1$ as above let $q(\alpha,\beta)\leq^* p_1\cat \langle \alpha,\beta\rangle$ witnessing that $\mathcal{N}_{q(\alpha,\beta)}\s O$. As before, let $a$ be the auxiliary condition whose first measure one set is the collection of those $\alpha$'s, the second measure one set is the collection of those $\beta$'s and the rest of measure one sets are obtained intersecting the measure one sets of the $q(\alpha,\beta)$'s.

     Arguing exactly as before one shows that  $\mathcal{N}_a\s O$. Contradiction.
   \end{proof}
   Let $p_2\leq^* p_1$ be the condition whose first measure one set is the collection of those $\alpha$'s and its second measure one set is $\bigcap_{\alpha<\kappa_0} X_{1,\alpha}$ where $X_{1,\alpha}$ denotes the complement of the above $\mathcal{U}_1$-small set. This is well-defined, by $\kappa_0^+$-completeness of $\mathcal{U}_1$. Proceeding in this fashion we obtain a $\leq^*$-decreasing sequence $\langle p_n\mid n<\omega\rangle$ such that for all $t\in \prod_{i=0}^n \vec{A}^{p_n}$, there is no $q\leq^* p_n\cat t$ such that $\mathcal{N}_q\s O$.

   \smallskip

   Let $p^*$ be a $\leq^*$-lower bound for the sequence $\langle p_n\mid n<\omega\rangle$. Note that because $p^*$ is a lower bound of the $p_n$ for no condition $r\leq p^*$, $\mathcal{N}_r \s O$.

   \begin{claim}
     $\mathcal{N}_{p^*}$ is disjoint from $O$. 
   \end{claim}
   \begin{proof}[Proof of claim]
    Suppose otherwise. Since $\mathcal{N}_{p^*}\cap O$ is open an non-empty there is $q\in \mathbb{P}(\vec{\mathcal{U}})$ such that $\mathcal{N}_q\s \mathcal{N}_{p^*}\cap O$. Then, there is $r\leq p^*,q$. But then  $\mathcal{N}_r \s \mathcal{N}_q\s O$, which is a contradiction.     
   \end{proof}

% Now assume that $O$ is, in addition, dense in the $\vec{\mathcal{U}}$-\textsf{EP} topology. By density,  $\mathcal{N}_{p^*}\cap O$ is non-empty -- say, this is witnessed by $x$. Let $q$ be a condition such that $x\in \mathcal{N}_q\s O$. Since $x\in \mathcal{N}_q\cap \mathcal{N}_{p^*}$ there is $r\leq q, p^*$ such that $\mathcal{N}_r\s \mathcal{N}_q\cap \mathcal{N}_{p^*}$. It thus follows that there is an extension $r$ of $p^*$ such that  $\mathcal{N}_r\s O$. But, one again, this contradicts the property of $p^*$. 

    So, we have showed that if $O$ is open then either there is a $\leq^*$-extension $q$ of $p$ such that $\mathcal{N}_q \s O$ or $\mathcal{N}_q\cap O=\emptyset$. Moreover, if $O$ is dense,    the second alternative cannot happen for $O$ will always intersect $\mathcal{N}_q$.
\end{proof}

\begin{lemma}\label{lemma: meager sets are Ramsey null}
    If $N\s \prod_{n<\omega}\kappa_n$ is a $\kappa$-meager set in $\vec{\mathcal{U}}$-$\textsf{EP}$  then for each $p\in \mathbb{P}(\vec{\mathcal{U}})$ there is $q\leq^* p$ such that $\mathcal{N}_q\cap N=\emptyset.$
\end{lemma}
\begin{proof}
Let $\{N_\xi\mid \xi<\kappa\}$ be a collection of $\UEP$-nowhere dense sets with $$\textstyle N=\bigcup_{\xi<\kappa} N_\xi.$$ Denote $O_\xi:=(\prod_{n<\omega}\kappa_n)\setminus N_\xi$. These are $\UEP$-dense open sets.  Fix $p\in \mathbb{P}(\vec{\mathcal{U}})$.

\smallskip

Let $\{I_n\mid n<\omega\}$ be a partition of $\kappa$ such that $|I_n|<\kappa_{\ell(p)+n}$ for all $n<\omega$. We define a $\leq^*$-decreasing sequence of conditions  $\langle p_n\mid n<\omega\rangle$ as follows.

For $n=0$ construe a $\leq^*$-decreasing sequence $\langle q_\xi\mid \xi< |I_0|\rangle$ below $p$ such that $\mathcal{N}_{q_\xi}\s O_\xi$ which  exists by combining  Lemma \ref{lemma: dense open sets are Ramsey} with  $|I_0|<\kappa_{\ell(p)}.$ Let $p_0$ be a $\leq^*$-lower bound for these $q_\xi$'s noticing that $\mathcal{N}_{p_0}\s O_\xi$ for all $\xi<|I_0|.$

Suppose that $p_n$ has been already defined. Let $\langle O_\xi\mid \xi<|I_{n+1}|\rangle$ be an enumeration of the dense open sets listed by indices in $I_{n+1}$. For each stem $t$ in $ \prod_{i=\ell(p)}^{\ell(p)+n+1} A^{p_n}_i$ we define a $\leq^*$-decreasing sequence of conditions $\langle q^t_\xi\mid \xi<|I_{n+1}|\rangle$  $\leq^*$-below $p_n\cat t$ such that $\mathcal{N}_{q^t_\xi}\s O_\xi$. After defining this sequence we let $q^t$ be a $\leq^*$-lower bound for it, which exists precisely because $|I_{n+1}|<\kappa_{\ell(p)+n+1}$. Now let $p_{n+1}\leq^* p_n$ be the condition whose measure one sets in $[\ell(p),\ell(p)+n+1]$ are those of $p_n$ and past that coordinate the measure one sets are the intersection (ranging over all possible $t$'s) of the measure one sets of the $q^t$'s. It is routine to check that $$\text{$\mathcal{N}_{p_{n+1}}\s O_{\xi}$ for all $\xi<|I_{n+1}|.$}$$

After this procedure we have produced a $\leq^*$-decreasing sequence of conditions $\langle p_n\mid n<\omega\rangle$ and we can take $q$ be a $\leq^*$-lower bound for it. Clearly, $q\leq^* p$. By construction $\mathcal{N}_{q}
\s O_\xi$ for all $\xi<\kappa$. Ergo, $\mathcal{N}_{q}\cap N=\emptyset$, as needed.
\end{proof}

\begin{comment}

\begin{lemma}\label{lemma: Open sets are completely ramsey}\ale{\tiny Need to correct. It should say $q\leq^* p$.}
    If  $O\s \prod_{n<\omega}\kappa_n$ is an $\vec{\mathcal{U}}$-$\textsf{EP}$ open set then for each $p\in \mathbb{P}(\vec{\mathcal{U}})$ there is $q\leq p$ such that either  $\mathcal{N}_q\s  O$ or $\mathcal{N}_q\cap O=\emptyset.$
\end{lemma}
\begin{proof}
If there is $q\leq p$ such that $\mathcal{N}_q\s O$ then we are done. So, suppose that this is not the case. Then,  we claim that $\mathcal{N}_p$ is disjoint from $O$.

Towards a contradiction, suppose that $\mathcal{N}_p\cap O$ is non-empty. Let $x\in \mathcal{N}_p\cap O$. Since this intersection is open, there is a condition $q\in \mathbb{P}(\vec{\mathcal{U}})$ such that  $x\in \mathcal{N}_q\s \mathcal{N}_p\cap O$.  Then, there is $r\leq p,q$ such that $x\in \mathcal{N}_r\s \mathcal{N}_q\s O$, thus contradicting our initial assumption.
\end{proof}
\end{comment}

\begin{lemma}\label{lemma: BP implies Ramsey}
    If $A\s \prod_{n<\omega}\kappa_n$ has the $\vec{\mathcal{U}}$-Baire property then for each $p\in \mathbb{P}(\vec{\mathcal{U}})$ there is $q\leq^* p$ such that either $\mathcal{N}_q\s A$ or $\mathcal{N}_q\cap A=\emptyset$.
\end{lemma}
\begin{proof}
Let $O$ and $N$ be, respectively, open and meager in the $\vec{\mathcal{U}}$-\textsf{EP} topology such that $A\triangle O\s N$. Fix $p\in \mathbb{P}(\vec{\mathcal{U}})$. By Lemma~\ref{lemma: meager sets are Ramsey null} there is $q\leq^* p$ such that $\mathcal{N}_q\cap N=\emptyset$. By Lemma~\ref{lemma: Open sets are completely ramsey} there is $r\leq^* q$ such that either $\mathcal{N}_r\s O$ or  $\mathcal{N}_r\cap O=\emptyset$. Note that $\mathcal{N}_r\s \mathcal{N}_q$ therefore it is disjoint from $N$ as well. This implies that either $\mathcal{N}_r\s A$ or $\mathcal{N}_r\cap A=\emptyset$ hold, thus completing the lemma.
\end{proof}

\begin{theorem}
    Let $\mathbb{P}^{\mathrm{diag}}$ be the diagonal Merimovich forcing and $G\s \mathbb{P}^{\mathrm{diag}}$ a generic for it.  Then, in $L(\mathcal{P}(\kappa))^{V[G]}$ every set $A\s \prod_{n<\omega}\kappa_n$ is $\vec{\mathcal{U}}$-completely Ramsey.
\end{theorem}
\begin{proof}
   This follows  combining Lemma~\ref{lemma: BP implies Ramsey} with \cite[Theorem~3.26]{DPT}.
\end{proof}

\subsection{J\'onsson cardinals} Recall that a cardinal $\lambda$ is \emph{J\'onsson} if for every function $F:[\lambda]^{<\omega}\to\lambda$ there is $H\s\lambda$ of order-type $\lambda$ such that $F``[H]^{<\omega}\subsetneq\lambda$
\begin{theorem}\label{thm: Jonsson}
    Suppose $\kappa<\lambda$ are respectively ${<}\lambda$-supercompact and J\'onsson. Then there is a forcing extension $W$ of $V$ such that $\kappa$ is a strong limit of countable cofinality in $W$ and $\HOD^W_{\mathcal{P}(\kappa)}\models``\kappa^+$ is J\'onsson".
\end{theorem}
\begin{proof}
         By doing a suitable preparatory class forcing we may assume that for every set-generic extension $W$ of $V$, $V\s\HOD^W$.\footnote{The idea of the preparation, which is essentially due to McAloon \cite{mcaloon1971consistency}, is that we will make the $\kappa$ and $\lambda$ Laver-indestructible, and then do $\lambda$-directed closed forcing above $\lambda$ to arrange that every set of ordinals is coded unboundedly often into the values of the continuum function.} %We may assume that for every set-generic extension $W$ of $V$, $V\s\HOD^W$ (see the proof of Theorem \ref{thm: def TP}). %Working in $V$, let $\mathbb{P}:=\mathbb{P}(\kappa,\lambda)$ denote the Merimovich's forcing, and let $\langle\mathbb{P}_d, \pi_{e,d}\colon \mathbb{P}_e\rightarrow\mathbb{P}_d\mid d,e\in \mathcal{D}\ \wedge\ d\sle e\rangle$ be its associated sequence of forcings and projections. 
        Let $G$ be $\mathbb{P}$-generic, and let $F:[\lambda]^{<\omega}\to\lambda$ be a function in $\HOD_{\mathcal{P}(\kappa)}^{V[G]}$. Then there is a first-order formula $\varphi(x,y,z,w)$ such that
    $F=\{\langle\vec{\alpha},\beta\rangle\in [\lambda]^{<\omega}\times\lambda\mid V[G]\models \varphi(\vec{\alpha},\beta,\gamma,x)\},$ for some parameters $\gamma\in\ord$ and $x\in\mathcal{P}(\kappa)^{V[G]}$. By the $\kappa$-capturing property, there is $d\in \mathcal{D}^\ast$ with $x\in V[\pi_d``G]$, so that $$F=\{\langle\vec{\alpha},\beta\rangle\in [\lambda]^{<\omega}\times\lambda\mid \exists p\in G\; p\Vdash^{V[G]}_{\mathbb{P}}\varphi(\check{\vec{\alpha}},\check{\beta},\check{\gamma},\dot{x})\},$$ where $\dot{x}$ is a $\mathbb{P}_d$-name for $x$. By Lemma \ref{lem: hom of the poset}, $$F=\{\langle\vec{\alpha},\beta\rangle\in [\lambda]^{<\omega}\times\lambda\mid \exists p\in \pi_d``G\; p\Vdash^{V}_{\mathbb{P}}\varphi(\check{\vec{\alpha}},\check{\beta},\check{\gamma}, \dot{x})\},$$ yielding $F\in V[\pi_d``G]$. Since $|\mathbb{P}_d|<\lambda$, $\lambda$ is J\'onsson in $V[\pi_d``G]$. So there is $H\in\mathcal{P}(\lambda)^{V[\pi_d``G]}$ of order-type $\lambda$ such that $F``[H]^{<\omega}\subsetneq\lambda$. To finish we have to verify that $H\in\HOD_{\mathcal{P}(\kappa)}^{V[G]}$. In fact, by assumption we have that $V\s\HOD^{V[G]}\s\HOD_{\mathcal{P}(\kappa)}^{V[G]}$. On the other hand, Lemma \ref{coding into subsets of kappa} leads to $\pi_d``G\in L(\mathcal{P}(\kappa))^{V[G]}\s\HOD_{\mathcal{P(\kappa)}}^{V[G]}$. As a result $H\in V[\pi_d``G]\s\HOD_{\mathcal{P}(\kappa)}^{V[G]}$. %As argued in the proof of Theorem \ref{thm: def TP}, we conclude that $H\in \HOD_{\mathcal{P}(\kappa)}^{V[G]}$.
\end{proof}
\begin{remark}\hfill
\begin{enumerate}
    \item Assuming $\lambda$ is Ramsey, the same proof strategy of Theorem \ref{thm: Jonsson} yields the consistency of $\HOD_{\mathcal{P}(\kappa)}\models``\kappa^+$ is Ramsey".
    \item With a similar argument, Gitik and Merimovich \cite{GitMer} showed that, under suitable large cardinal assumptions, there are models $W_0$ and $W_1$ of $\zfc$ such that $\HOD^{W_0}_{\mathcal{P}(\kappa)}\models``\kappa^+$ is measurable" and $\HOD^{W_1}_{\mathcal{P}(\kappa)}\models\axiomfont{TP}_{\kappa^+}$. 
\end{enumerate}
    
\end{remark}

\subsection{Measurability}
In this section we assume that $\lambda$ is a measurable cardinal above $\kappa$ (as witnessed by $U$) and that $\mathbb{P}$ is the corresponding Merimovich's forcing. 

In \cite{AIM} Cummings et al. showed that a successor of a singular cardinal $\kappa$ can be measurable in $\HOD_x$ for every $x\s \kappa$. The large cardinal assumptions of the main result of \cite{AIM} were later improved by Gitik--Merimovich \cite{GitMer} who utilized Merimovich's forcing to establish the same result. Here we show that a similar configuration can be obtained inside a model that contains $\mathcal{P}(\kappa)$ of the universe. Specifically, we prove the following:
    \begin{theorem}\label{thm: measurability}
        $\kappa^+$ is measurable in $L(\mathcal{P}(\kappa)^{V[G]},U)$.
    \end{theorem}
    \begin{proof}
  Working in  $L(\mathcal{P}(\kappa)^{V[G]},U)$ define the filter $$U^\ast:=\{Y\subseteq\lambda\mid\exists X\in U\, (X\s Y)\}.$$ For each domain $d\in\mathcal{D}^\ast$ define $U_d:=\{Y\in\mathcal{P}(\lambda)^{V[\pi_{d}``G]}\mid \exists X\in U\, (X\s Y)\}$.   Note that $V[\pi_{d}``G]\models``U_d\text{ is a measure on $\lambda$''}$ simply because $|\mathbb{P}_d|<\lambda$. We claim that $L(\mathcal{P}(\kappa)^{V[G]},U)\models``U^\ast\text{ is a measure on $\kappa^+$''}$. 

  \begin{claim}
      $U^\ast $ is an ultrafilter in $L(\mathcal{P}(\kappa)^{V[G]},U)$.
  \end{claim}
  \begin{proof}[Proof of claim]
       Let $X\in \mathcal{P}^{V[G]}(\lambda)\cap L(\mathcal{P}(\kappa)^{V[G]},U)$ be with $X\notin U^\ast$. Clearly, $Z:=\lambda\setminus X\in L(\mathcal{P}(\kappa)^{V[G]},U)$ and $Z\s V$. So we may apply Corollary \ref{cor: L(P(kappa))-capturing} to find some $d\in \mathcal{D}^*$ with $Z\in V[\pi_{d}``G]$. By assumption $X\notin U_d$, and since $U_d$ is an ultrafilter it must be the case that $Z\in U_d$, yielding $Z\in U^\ast$.
  \end{proof}

\begin{claim}
 $U^\ast\text{ is $\kappa^+$-complete''}$  in $L(\mathcal{P}(\kappa)^{V[G]}, U)$.  
\end{claim}
\begin{proof}[Proof of claim]
   Fix $\mu<\lambda$, and let $\vec{X}:=\langle X_\alpha\mid\alpha<\mu\rangle\in L(\mathcal{P}(\kappa)^{V[G]}, U)$ be a sequence consisting of elements in $U^\ast$. For each $\alpha<\lambda$, let $d_\alpha\in\mathcal{D}^\ast$ be such that $X_\alpha\in V[\pi_{d_\alpha}``G]$. Since $\mu<\lambda$, $d:=\bigcup_{\alpha<\mu}d_\alpha$ is in $\mathcal{D}^\ast$, and so $\vec{X}\s V[\pi_{d}``G]$. By Corollary \ref{cor: L(P(kappa))-capturing}, $\vec{X}\in V[\pi_{e}``G]$ for some $e\in\mathcal{D}^\ast$ with $d\s e$. Since $V[\pi_{e}``G]\models``U_e\text{ is a measure on $\lambda$''}$ and $\{X_\alpha\mid \alpha<\mu\}\s U_e$, it follows  $\bigcap_{\alpha<\mu}X_\alpha\in U_e\s U^\ast$. 
   
\end{proof}

    \begin{claim}
         $U^\ast\text{ is normal}$  in $L(\mathcal{P}(\kappa)^{V[G]}, U)$.
    \end{claim}
    \begin{proof}[Proof of claim]
     Suppose that $$L(\mathcal{P}(\kappa)^{V[G]}, U)\models``f\colon\lambda\rightarrow\lambda\text{ is a regressive function''.}$$ Then $f\in V[\pi_{d}``G]$ for some $d\in\mathcal{D}^\ast$, which exists by Corollary \ref{cor: L(P(kappa))-capturing}. Since $V[\pi_{d}``G]\models``U_d\text{ is a measure on $\lambda$''}$, it must be the case that $V[\pi_{d}``G]\models``f$ is constant on some $X\in U\s U^\ast$''. 
    \end{proof}
We are done with the proof of the theorem.
\end{proof}

Combining the above with ideas of Ben-Neria--Hayut \cite[\S7]{BenHay} we can prove:

\begin{theorem}
    There is a generic extension $V[G\ast H]$ of $V[G]$ via a $\lambda$-distributive forcing where $\mathrm{Cub}_\lambda\restriction (E^\lambda_{\mathrm{Reg}})^V$ is a $L(\mathcal{P}(\kappa),U)^{V[G\ast H]}$-ultrafilter.

    Moreover, $U^*$ is the set of all $V[G\ast H]$-stationary sets $X\s (E^\lambda_{\mathrm{Reg}})^V$ in $L(\mathcal{P}(\kappa),U)^{V[G\ast H]}$.
\end{theorem}
\begin{proof}
    Let $U^*$ be the measure on $\lambda$ in $L(\mathcal{P}(\kappa),U)^{V[G]}$ identified in Theorem \ref{thm: measurability}. 

    Note that the filter $\mathcal{F}$ generated in $V[G]$ by the set $\{A\cup \mathrm{Sing}\mid A\in U^*\}$ is precisely the filter generated in $V[G]$ by what Ben-Neria--Hayut call $\bar{\mathcal{U}}$. Using the Strong Prikry Property of $\mathbb{P}$ this filter is $\lambda$-complete (see \cite[Lemma 7.2]{BenHay}).

    Now, over $V[G]$, force with the club shooting poset $\mathbb{C}_{\mathcal{F}}$: Conditions are pairs $\langle c, A\rangle$ consisting of a closed bounded subset of $\lambda$ and $A\in \mathcal{F}$, and the ordering is defined by stipulating  $\langle c, A\rangle\leq \langle d, B\rangle$ iff $d\sqsubseteq c$ and $A\s B$ and $c\setminus \max(d)+1\s B$. Crucially, $\mathbb{C}_{\mathcal{F}}$ is $\lambda$-distributive over $V[G]$ (see \cite[Lemma 7.3]{BenHay}) so it does not add new subsets to $\kappa$. This poset introducec a generic club $C\s \lambda$ that diagonalizes $\mathcal{F}$; to wit, given any $A\in \mathcal{F}$, $C\s^* A$. In particular, $C\cap (E^\lambda_{\mathrm{Reg}})\s^* B$ for all $B\in U^*$.

    \smallskip

    Working in $V[G\ast H]$ we have 
    $$L(\mathcal{P}(\kappa))^{V[G\ast H]}=L(\mathcal{P}(\kappa))^{V[G]}.$$
   So $U^*$ is still a measure on $\lambda$ therein.

   \smallskip

    Given $X\in L(\mathcal{P}(\kappa))^{V[G\ast H]}$ a subset of $(E^\lambda_{\mathrm{Reg}})^V$ we have that either $X\in U^*$ or $\lambda\setminus X\in U^*$. Since $(E^\lambda_{\mathrm{Reg}})^V\in U\s U^*$ it follows that either $X\in U^*$ or $(E^\lambda_{\mathrm{Reg}})^V\setminus X\in U^*$. As a result, $C\cap (E^\lambda_{\mathrm{Reg}})^V$ is almost contained in  either $X$ or $(E^\lambda_{\mathrm{Reg}})^V\setminus X$. This proves the claim about the ultrafilter.

    The moreover claim is proved analogously.
   \end{proof}
\begin{cor}
    The following properties hold in $L(\mathcal{P}(\kappa)^{V[G]}, U)$:
    \begin{enumerate}
        \item Every subset of ${}^\omega \kappa$ has the $\kappa$-$\psp$.
        \item There are no scales at $\kappa$.
        \item Both $\sch_\kappa$ and $\diamondsuit_{\kappa^+}$ fail.
        \item There is no $\kappa^+$-sequence consisting of distinct members of $\mathcal{P}(\kappa)$.
        %\item Both $\ads_\kappa$ and $\mathrm{NPT}(\kappa^+,\aleph_1)$ fail.
        \item $\AP_\kappa$ fails.
        \item  $\kappa\xrightarrow[]{\mathrm{OD}} (\omega)^\omega_{V_\mu}$ holds for all $\mu<\kappa$.
        \item $\kappa^+$ is measurable.
        \item The tree property at $\kappa^+$ holds.
    \end{enumerate}
\end{cor}
\begin{proof}
    (1) It follows from the same argument given in the proof of Theorem \ref{PSP in Solovay}. The reason is that $U$ is a parameter in the ground model $V$.

    (2)--(4) Follow from (1) (see Proposition \ref{thm: combinatorics after psp}).

    (5) and (6) Follow from the very same arguments given in Theorems \ref{thm: not ap} and \ref{thm: kafkoulis} because $U$ is a parameter in $V$.

    (7) This is Theorem \ref{thm: measurability}.

    (8) This is a standard argument. See for instance \cite[Theorem 4]{ShiTrang}.
\end{proof}

\section{Open Questions}\label{sec: open questions}

In \cite{ShiTrang}, Shi--Trang established most of the combinatorial properties considered in this paper from the assumption of $I_0(\kappa)$. In fact, the crucial ingredient in many of their arguments is the consequence of $I_0(\kappa)$ that $\kappa^+$ is measurable in $L(\mathcal{P}(\kappa))$. In the present work, we have derived the corresponding combinatorial properties from the assumption of $\kappa$-$\psp$. This naturally leads to the following question.
\begin{question}
    What is the relation between:
    \begin{enumerate}
        \item $\kappa^+$ is measurable.
        \item Every subset of ${}^\omega \kappa$ has the $\kappa$-$\psp$. 
    \end{enumerate}
\end{question}
Another interesting combinatorial property whose status in our model remains unclear is stationary reflection at $\kappa^+$, namely $\refl(\kappa^+)$. In \cite{ShiTrang}, the authors established this property in $L(\mathcal{P}(\kappa))$ from a strengthening of $I_0(\kappa)$ implying that
$$
L(\mathcal{P}(\kappa))\models \text{``$\kappa^+$ is $\mathcal{P}_{\kappa^+}(\mathcal{P}(\kappa))$-supercompact.''}
$$
In the case $\kappa=\omega$, Trang showed that if $\kappa$ is supercompact, then forcing with $\Col(\omega,{<}\kappa)$ yields a model in which
$
L(\mathcal{P}(\mathbb{R}))
$
satisfies that $\omega_1$ is $\mathbb{R}$-supercompact. We do not know how to obtain an analogous conclusion in our setting, since Trang's argument relies on lifting supercompactness embeddings through the Lévy collapse. This leads naturally to the following question.
\begin{question}Working in $V[G]$:
\begin{enumerate}
    \item Does $L(\mathcal{P}(\kappa))\models ``\refl(\kappa^+)$ holds''?
    \item  Can $\kappa^+$ be $\mathcal{P}_{\kappa^+}(\mathcal{P}(\kappa))$ supercompact in $L(\mathcal{P}(\kappa))$?
\end{enumerate}
\end{question}
Another natural question related to our findings on partition relations is:
\begin{question}
    Does $``\kappa\to  (\omega)^\omega_{V_\mu}$ holds for all $\mu<\kappa$'' hold in $L(\mathcal{P}(\kappa))^{V[G]}$?
\end{question}
More specifically, the question is whether one can find homogeneous sets for colorings that are definable from parameters in $\mathcal{P}(\kappa)$. This appears to be substantially more difficult, since the construction of such homogeneous sets seems to require modifying some of the Prikry sequences relative to a non-normal measure
\[
U_\alpha=\{X\subseteq\kappa\mid \alpha\in j(X)\}
\]
for $\alpha>\kappa$. The difficulty is that the Prikry sequences added by $\mathbb{P}$ are not independent of one another, so altering one sequence typically necessitates corresponding modifications to several others as well.

%\begin{question}[Asperó]
   % Is it consistent that many/all cardinals $\kappa$ of countable cofinality satisfy that $L(\mathcal{P}(\kappa))$ is a $\kappa$-Solovay model? Could some form of Radin forcing be useful towards that sort of ``Solovay maximum''?
%\end{question}
%\begin{question}
 %   What is the exact consistency strength of $$``\zf+\exists\kappa=\beth_\kappa\, (\dc_\kappa\ \wedge\ \cf(\kappa)=\omega\ \wedge\ \neg\AP_{\kappa})\text{''?}$$
%\end{question}
%Since the club filter is not an ultrafilter,\seba{prove or remark this} the following questions are challenging. 
%\begin{question}
%    Is $\kappa^+$ measurable in $\kappa$-Solovay models?
%\end{question}

\smallskip

Finally, a question that the current method is not able to address is:
\begin{question}
   Is $\HOD_{\mathcal{P}(\aleph_\omega)}\models``\aleph_{\omega+1} \text{ is measurable"}$ consistent? Similarly, can $\aleph_{\omega+1}$ be measurable in $L(\mathcal{P}(\aleph_\omega))$?
\end{question}

\medskip

\emph{\small Acknowledgments. This work was partially supported by the Simons Foundation grant (award no. SFI-MPS-T-Institutes-00010825) and from State Treasury funds as part of a task commissioned by the Minister of Science and Higher Education under the project “Organization of the Simons Semesters at the Banach Center - New Energies in 2026-2028” (agreement no. MNiSW/2025/DAP/491). The second author's work is funded by the National Science Center, Poland under the Weave-Unisono Call, registration number UMO-2023/05/Y/ST1/00194.}

\begin{center}   
        \includegraphics[width=0.8\linewidth]{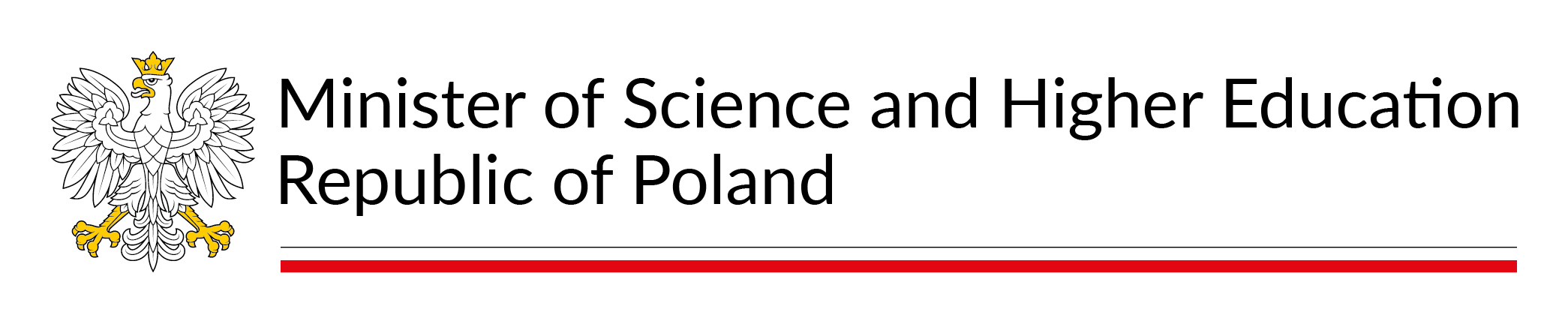}
       % \caption{Caption}
        %\label{fig:enter-label}
    \end{center}

    \begin{center}   
        \includegraphics[width=0.8\linewidth]{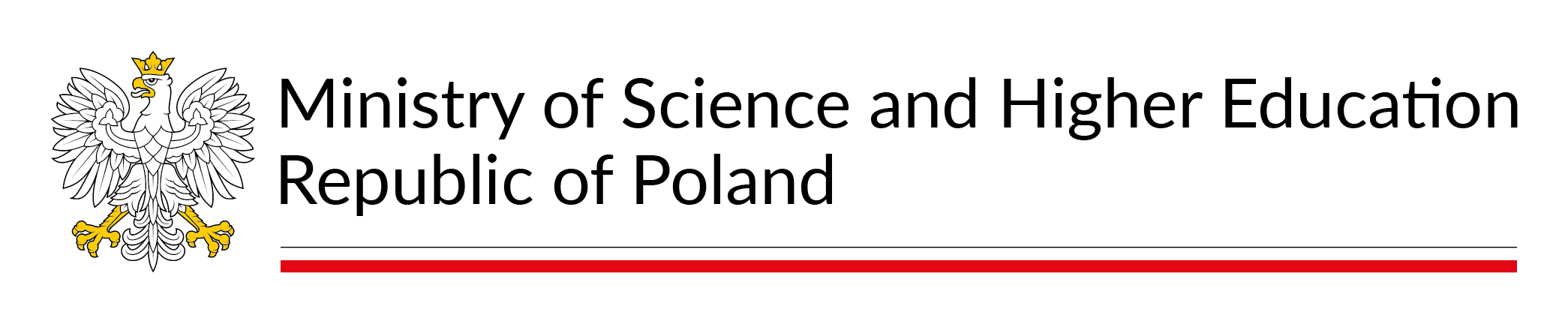}
       % \caption{Caption}
        %\label{fig:enter-label}
    \end{center}

\bibliographystyle{alpha} 
\bibliography{biblio}

\end{document}